\documentclass[12pt]{amsart}
\usepackage[margin=1.1 in]{geometry}
\usepackage{amssymb}
\usepackage{csquotes}
\usepackage{enumitem}
\usepackage{xcolor}
\usepackage{hyperref}
\usepackage{cite}

\theoremstyle{plain}
\newtheorem{theorem}{Theorem}[section]
\newtheorem{lemma}[theorem]{Lemma}
\newtheorem{proposition}[theorem]{Proposition}
\newtheorem{corollary}[theorem]{Corollary}

\theoremstyle{definition}
\newtheorem{definition}[theorem]{Definition}

\newtheorem{remark}[theorem]{Remark}
\newtheorem{example}[theorem]{Example}

\numberwithin{equation}{section}

\makeatletter
\@namedef{subjclassname@2020}{\textup{2020} Mathematics Subject Classification}
\makeatother

\title[The condition $(\Omega)$ for weighted spaces of holomorphic functions]{A linear topological invariant for weighted spaces of holomorphic functions 
}
\author{Andreas Debrouwere, Quinten Van Boxstael}
\address{A. Debrouwere, Department of Mathematics and Data Science \\ Vrije Universiteit Brussel, Belgium\\ Pleinlaan 2 \\ 1050 Brussels \\ Belgium}
\email{Andreas.Debrouwere@vub.be}

\address{Q. Van Boxstael,  Department of Mathematics and Data Science \\ Vrije Universiteit Brussel, Belgium\\ Pleinlaan 2 \\ 1050 Brussels \\ Belgium}
\email{Quinten.Van.Boxstael@vub.be}


\begin{document}
\begin{abstract}We study the linear topological invariant \((\Omega)\) for a class of Fr\'echet spaces of holomorphic functions of rapid decay on strip-like domains in the complex plane, defined via weight function systems. We obtain a complete characterization of the property \((\Omega)\) for such spaces in terms of an explicit condition on the defining weight function systems. As an application,  we investigate the surjectivity of the Cauchy-Riemann operator on certain weighted spaces of vector-valued smooth functions. 
\end{abstract}

\maketitle

\section{Introduction}
 The splitting theory for Fr\'echet spaces \cite{Meise-Vogt,V-FuncExt1Frechet} is one of the major tools of modern abstract functional analysis, with numerous applications including isomorphic classification and sequence space representations for  function spaces, the (non)-existence of linear continuous extension operators on spaces of Whitney jets, and, solvability of vector-valued and parameter dependent PDE's. The hypotheses in this theory are typically formulated in terms of certain   $(\Omega)$ and/or $(DN)$-type conditions. Hence, verifying these linear topological invariants for concrete function spaces becomes an important issue, often leading to interesting analytic problems.
 
 The goal of this article is to study the condition $(\Omega)$ for a class of weighted  Fr\'echet spaces of holomorphic functions. 
This and related properties have been extensively studied for various types of (weighted) spaces of holomorphic functions, see e.g.~\cite{Sequence space representations, Kruse1,Kruse2, Langenbruch2012, Langenbruch2016, Meise, Momm, Hadamard}. As an application, we investigate the surjectivity of the Cauchy-Riemann operator $$\displaystyle\overline{\partial} = \frac{1}{2}\left(\frac{\partial}{\partial x}+i\frac{\partial}{\partial y}\right)$$  on certain weighted spaces of vector-valued smooth functions; see \cite{Kruse1,Kruse2} for earlier work in this direction. This serves as the main motivation of our work on the property $(\Omega)$. 
 
 We now introduce the weighted spaces of smooth and holomorphic functions we shall be concerned  with in this article, originally defined in \cite{Surjectivity}. By a  \emph{generalized strip} we mean a subset of the complex plane of the form
 $$
T^{F,G} = \{z \in \mathbb{C} \mid -G(\mathrm{Re \ }z) < \mathrm{Im \ }z < F(\mathrm{Re \ }z)\},
 $$
where $F,G: \mathbb{R} \to (0,\infty)$ are uniformly continuous and satisfy  \(0< \inf_{t\in \mathbb{R}} F(t) \leq \sup_{t\in \mathbb{R}} F(t)<\infty \) and  \(0< \inf_{t\in \mathbb{R}} G(t) \leq \sup_{t\in \mathbb{R}} G(t)<\infty \). A \emph{weight function} is a non-decreasing unbounded continuous function $w: [0,\infty) \to [0,\infty)$, while a \emph{weight function system} is a pointwise non-dereasing sequence $W = (w_N)_{N \in \mathbb{N}}$ of weight functions. We define \(\mathcal{K}_{W}(T^{F,G})\) as the Fr\'echet space  consisting of all \(f \in C^{\infty}(T^{F,G})\) such that
 \[
 \sup_{z\in \overline{T^{aF,aG}}} e^{w_N(\lvert \mathrm{Re}z\rvert)} |f^{(\alpha)}(z)|<\infty \qquad \forall N \in  \mathbb{N}, \alpha \in \mathbb{N}^2, a \in (0,1).
 \]
and $\mathcal{U}_W(T^{F,G})$ as its subspace consisting of holomorphic functions on $T^{F,G}$.  

In our  first main result  (Theorem \ref{theorem Omega}),  we characterize the property \((\Omega)\) for  $\mathcal{U}_W(T^{F,G})$ in terms of an explicit condition on $W$ (under mild assumptions on $W$). A widely used class of weight function systems in the literature consists of those of the form $\tilde{W}_w = (Nw)_{N \in \mathbb{N}}$, where \(w \colon [0,\infty) \rightarrow  [0,\infty)\) is a weight function.
For such weight function systems, we obtain that, if \(w(2t) = O(w(t))\) as $t  \to \infty$ and
$$
\int_0^{\infty} e^{-Cw(t)}\text{d}t < \infty
$$
for some $C >0$, then  $\mathcal{U}_{\tilde{W}_w}(T^{F,G})$ always satisfies $(\Omega)$ (Corollary \ref{omega for N outside}).
  This  extends an earlier result \cite{Sequence space representations} of the first author,  where only  horizontal strips $T_h = \mathbb{R} + i(-h,h)$, $h >0$, were considered.  We also mention the work of Kruse \cite{Kruse2} dealing with sets of the form $T^h \backslash K$, $h >0$ and $K \subseteq \overline{\mathbb{R}}$ compact, but under much more restrictive conditions on the weight  function $w$.

The condition \((\Omega)\)  can be be considered either as a quantitative approximation property with respect to the seminorms  of the Fr\'echet space under consideration or as a log-convex inequality for its  dual seminorms. In \cite{Sequence space representations} the first point of view was taken and the approximation property was shown by using  a technique, developed by Langenbruch \cite{Langenbruch2012},  combining so-called holomorphic cut-off functions and  Hörmander's \(L^2\)-method.  However, due to the use of the \(L^2\)-method, this approach only seems to work for horizontal strips. We develop here a new technique that is applicable to the much broader class of generalized strips. To this end, we view \((\Omega)\) as a log-convex inequality for the  dual seminorms.

Our proof consists of two main steps.  First, we give a concrete description of the dual of  $\mathcal{U}_W(T^{F,G})$ by weighted spaces of holomorphic functions.  This is much inspired by the analytic representation theory for  Fourier (ultra)hyperfunctions \cite{debrouwere2018non,HSP,Kruse2}.
Next, we show  a log-convex inequality for certain weighted supremum norms of holomorphic functions, reminiscent of the Hadamard three-lines theorem.
The condition $(\Omega)$ for $\mathcal{U}_W(T^{F,G})$ then follows by applying this inequality to the spaces representing  the dual of  $\mathcal{U}_W(T^{F,G})$. 

We mention that the conditions on the weight function systems imposed in our work are more restrictive  than in Langenbruch's method \cite{Langenbruch2012, Sequence space representations}.

 In \cite{Surjectivity}, we  characterized the surjectivity of the map 
 \begin{equation}
\label{sv-surj}
\displaystyle\overline{\partial} : \mathcal{K}_W(T^{F,G}) \to \mathcal{K}_W(T^{F,G})
\end{equation}
 in terms of a growth condition on $W$. Given a locally convex space $E$, we denote by  $\mathcal{K}_W(T^{F,G};E)$ the natural $E$-valued version of $\mathcal{K}_W(T^{F,G})$. Assuming that the map \eqref{sv-surj} is surjective, we study in this paper the surjectivity of the $E$-valued Cauchy-Riemman operator 
\begin{equation}
\label{vv-surj}
\displaystyle\overline{\partial} : \mathcal{K}_W(T^{F,G};E) \to \mathcal{K}_W(T^{F,G};E).
\end{equation}
If  $E$ is a Fr\'echet space, the map \eqref{vv-surj} is always surjective, as follows from the general theory of topological tensor products; see Remark \ref{Frechet}$(iii)$ for further details. We focus here on locally convex spaces $E$ having a countable fundamental family of bounded sets (such as $(DF)$-spaces), for which the surjectivity problem is more subtle. 
 
 Assume that $ \mathcal{U}_W(T^{F,G})$ satisfies $(\Omega)$ and let $E$ be a locally complete locally convex Hausdorff space having a countable fundamental family of bounded sets. In our second main result (Theorems \ref{Vector-valued surjectivity} and \ref{Converse}), we show that the map \eqref{vv-surj} is surjective if and only if $E$ satisfies the condition $(A)$ \cite{V-VektorDistrRandHolomorpherFunk} (Definition \ref{DEFA}).  This can be viewed as a weighted version of  a classical result of Vogt \cite{V-VektorDistrRandHolomorpherFunk},  which states that the map
 $$
\displaystyle\overline{\partial} : C^\infty(X;E) \to  C^\infty(X;E) 
$$
$X \subseteq \mathbb{C}$ open,  is surjective if and only $E$ satisfies $(A)$. 

 The condition $(A)$ is closely related to the property $(DN)$ for Fr\'echet spaces. For instance, a reflexive $(DF)$-space satisfies $(A)$ if and only if its strong dual satisfies $(DN)$. Hence, duals of power series spaces of infinite type  satisfy $(A)$, whereas  duals of power series spaces of finite type do not. Consequently, the space of tempered distributions $\mathcal{S}'(\mathbb{R}^d)$ satisfies $(A)$ \cite[Corollary 31.14]{Meise-Vogt}, while  the space of holomorphic germs on  a compact set $K$ in $\mathbb{C}^d$ that is the closure of a  bounded Reinhardt domain does not \cite[Theorem 5.5]{THD}. 
 
As an application of our second main result, we will show that the map
$$
\displaystyle\overline{\partial} : \mathcal{K}_W(T^{F,G};\mathcal{A}(\mathbb{R})) \to \mathcal{K}_W(T^{F,G};\mathcal{A}(\mathbb{R})),
$$
where $\mathcal{A}(\mathbb{R})$ denotes the space of real analytic functions on $\mathbb{R}$, is not surjective (Corollary \ref{real analytic}). This means that the Cauchy-Riemann equation, with real analytic parameter dependence,  is not solvable in $\mathcal{K}_W(T^{F,G})$.

This paper is organized as follows. We start with a preliminary section in which we introduce the weighted spaces of smooth and holomorphic functions we shall work with. Moreover, we state here the necessary background on linear topological invariants. 
Next, in Section \ref{sect-main results},  we state our main results and discuss some examples. Our characterization of  \((\Omega)\) for the spaces \(\mathcal{U}_W(T^{F,G})\) is proven in Section \ref{proof theorem omega}. Finally, in Section \ref{Proof converse}, we show our results about the surjectivity of the Cauchy-Riemann operator on the vector-valued spaces  \(\mathcal{K}_W(T^{F,G;E})\),
under a mild condition on \(W\).

\section{Preliminaries}
In this preliminary section, we fix the notation and introduce the weighted spaces of smooth and holomorphic functions  we shall be concerned with. These spaces were originally introduced in our earlier work \cite{Surjectivity}. 
Additionally, we provide the necessary background on linear topological invariants for locally convex spaces.

\subsection{Notation}

Let $E$ be a lcHs (=  locally convex Hausdorff space). We denote by $E^{\prime}$ the dual of $E$.  Given an absolutely convex bounded subset  $B$ of $E$,  we denote by $E_{B}$ the subspace of $E$ spanned by $B$ endowed  with the topology generated by the Minkowski functional of $B$. Since $E$ is Hausdorff, $E_{B}$ is normed. The set $B$ is called a \emph{Banach disk} if $E_B$ is complete. $E$ is said to be \emph{locally complete} if every bounded subset of $E$  is contained in a Banach disk. See \cite[Proposition 5.1.6]{BP} for various characterizations of this condition. Every sequentially complete lcHs is locally complete \cite[Corollary 23.14]{Meise-Vogt}.

Given two lcHs $E$ and $F$,  we denote by $L(E,F)$ the space consisting of all continuous linear maps from $E$ into $F$. We write $T^t: F' \to E'$ for the transpose of a map \(T\in L(E,F)\). 
The $\varepsilon$-product \cite{K-Ultradistributions3,S-TheorieDistValeurVect} of $E$ and $F$ is defined as
	\[ E \varepsilon F = L(F^{\prime}_{c}, E) , \]
where the subscript $c$ indicates that we endow $F^{\prime}$ with the topology of uniform convergence on absolutely convex compact subsets of $E$.
Let $E_1,E_2, F_1, F_2$ be lcHs. For $R \in L(E_1,E_2)$ and $S \in L(F_1,F_2)$ we define the map
\[ R \varepsilon S: E_1 \varepsilon F_1 \rightarrow E_2 \varepsilon F_2, \ A \mapsto   R \circ A \circ S^{t}. \]
\subsection{Weighted spaces of smooth and holomorphic functions.}
We start by introducing generalized strips.
\begin{definition}\label{gstrips}
    We write \(\mathcal{F}(\mathbb{R})\) for the family consisting of all functions \(F \colon \mathbb{R} \rightarrow [0,\infty)\) that are uniformly continuous and satisfy \(0<\inf_{t\in\mathbb{R}}F(t)\leq \sup_{t\in \mathbb{R}} F(t) < \infty\).  We define  \(\mathcal{F}_1(\mathbb{R})\) as the family consisting of all \(F \in \mathcal{F}(\mathbb{R}) \cap C^1(\mathbb{R})\) with  $\sup_{t\in \mathbb{R}} |F'(t)| < \infty$.
    Given \(F,G \in \mathcal{F}(\mathbb{R})\), we define  \emph{the generalized strip $T^{F,G}$} as 
    \[
    T^{F,G} = \{z\in\mathbb{C} \mid -G(\mathrm{Re} z) < \mathrm{Im} z < F(\mathrm{Re}z)\}.
    \]
     For \(F = G = h \in (0,\infty)\), we simply write 
    \[
    T^h = \{z\in \mathbb{C} \mid -h<\mathrm{Im}z < h\}.
    \]
\end{definition}

Next, we state a separation lemma. We need the following definition.
    \begin{definition}
  Let \(F,G \in \mathcal{F}(\mathbb{R})\). We write   \(F < G\)  to indicate that there is $\varepsilon >0$ such that 
  \[\{t+iF(t)\mid t \in \mathbb{R}\} + \overline{B}(0,\varepsilon)\subseteq \{z\in\mathbb{C}\mid  \mathrm{Im}z < G(\mathrm{Re }z)\}.\]
  Given  $\Phi \in \mathcal{F}(\mathbb{R})$, we write \(F < \Phi < G\) if \(F < \Phi\) and \(\Phi  < G\). If this holds, there is $\varepsilon >0$ such that
  \[\{t+i\Phi(t)\mid t \in \mathbb{R}\} + \overline{B}(0,\varepsilon)\subseteq \{z\in\mathbb{C}\mid F(\mathrm{Re } z)<  \mathrm{Im}z < G(\mathrm{Re }z)\}.\]
\end{definition}

\begin{lemma}\label{lipschitz} \cite[Lemma 2.3]{Surjectivity}\footnote{In fact, the function $\Phi$ can be chosen in \(\mathcal{F}(\mathbb{R})\cap C^{\infty}(\mathbb{R})\) with all derivatives bounded, but we will not need this.}
Let \(F \in \mathcal{F}(\mathbb{R})\) and \(a \in (0,1)\). There exists $\Phi \in \mathcal{F}_1(\mathbb{R})$
such that $aF <  \Phi < F$. 
\end{lemma}
\begin{remark}\label{stupid-remark}
$(i)$ Lemma \ref{lipschitz} particularly shows that $aF < F$ for all \(F \in \mathcal{F}(\mathbb{R})\) and \(a \in (0,1)\). \\
$(ii)$  Let \(F,G \in \mathcal{F}(\mathbb{R})\) with  \(F <G\). Then, there is $a \in (0,1)$ such that $F(t) \leq aG(t)$ for all $t \in \mathbb{R}$.
\end{remark}
Next, we discuss weight function systems.
\begin{definition}
 A function \(w \colon [0,\infty) \rightarrow  [0,\infty)\) is said to be a \emph{weight function} if it is non-decreasing and unbounded. A pointwise non-decreasing sequence  \( W = (w_N)_{N \in \mathbb{N}}\)  of weight functions is called a \emph{weight function system}.  
\end{definition}
We shall use the following conditions on a weight function system \(W = (w_N)_{N \in \mathbb{N}}\):
\textbf{\begin{enumerate}
    \item [\((\alpha)\)] \(\forall N \in \mathbb{N} \, \exists M \geq N, A > 1 \, \forall t\geq 0 : w_N(2t) \leq w_M(t) + \log A\).  
      \\
    \item [\((\epsilon)_0\)] \(\forall N \in \mathbb{N},  \mu > 0: \displaystyle{\int_0^{\infty} w_N(t) e^{-\mu t}\mathrm{d}t < \infty}\).
  \item [ \((N)\)] \(\forall N \in \mathbb{N}\, \exists M \geq N: \displaystyle{\int_{0}^{\infty} e^{w_N(t) - w_M(t)} \mathrm{d}t< \infty}\).
  \end{enumerate}}
The conditions $(\alpha)$ and $(N)$ are of a technical nature, while \((\epsilon)_0\) plays a more fundamental role; see Proposition \ref{nuclear} and Theorem \ref{Summary surjectivity} below. 

We are  ready to introduce the weighted spaces of scalar-valued smooth and holomorphic functions we shall deal with. Given \(X \subseteq \mathbb{C}\) open, we denote by \(\mathcal{H}(X)\) the space of  holomorphic functions on \(X\). 
\begin{definition}
 Let \(F, G \in \mathcal{F}(\mathbb{R})\) and let \(W =  (w_N)_{N \in \mathbb{N}}\) be a weight function system. We define $\mathcal{K}_{W}(T^{F,G})$ as the space consisting of all $\varphi \in C^{\infty}(T^{F,G})$ such that 
$$
\|\varphi\|_{N,\alpha,a}=\sup_{\xi \in \overline{T^{a F, a G}}} e^{w_N(\lvert \mathrm{Re \ }\xi\rvert)} |\varphi^{(\alpha)}(\xi)| < \infty, \qquad N \in \mathbb{N}, \alpha \in \mathbb{N}^2, a \in (0,1).
$$
We endow $\mathcal{K}_{W}(T^{F,G})$ with the locally convex topology generated by the system of seminorms $\{ \| \, \cdot \, \|_{N,\alpha,a} \mid N \in \mathbb{N}, \alpha \in \mathbb{N}^2, a \in (0,1) \}$.  We set
$$
\mathcal{U}_{W}(T^{F,G}) = \mathcal{K}_{W}(T^{F,G}) \cap \mathcal{H}(T^{F,G})
$$
and endow it 
 with the subspace topology induced by $\mathcal{K}_{W}(T^{F,G})$. 
 Note that \(\mathcal{K}_W(T^{F,G})\) and \(\mathcal{U}_W(T^{F,G})\) are both Fr\'echet spaces.
\end{definition}
We will also employ the following weighted space of entire functions.
\begin{definition}
Let \(W =  (w_N)_{N \in \mathbb{N}}\) be a weight function system. We define
$$
\mathcal{U}_W(\mathbb{C}) = \bigcap_{h > 0} \mathcal{U}_W(T_h).
$$
\end{definition}

We now give a useful derivative-free characterization of the Fr\'{e}chet space \(\mathcal{U}_W(T^{F,G})\).

\begin{lemma}\label{lemma-df}\cite[Lemma 2.8]{Surjectivity}
 Let \(W = (w_N)_{N \in \mathbb{N}}\) be a weight function system that satisfies \((\alpha)\) and let \(F, G \in \mathcal{F}(\mathbb{R})\). A function $\varphi \in \mathcal{H}(T^{F,G})$ belongs to 
 $\mathcal{U}_W(T^{F,G})$ if and only if 
$$
\|\varphi\|_{N,a}=  \sup_{\xi \in \overline{T^{a F, a G}}} e^{w_N(\lvert \mathrm{Re \ }\xi\rvert)} \lvert \varphi(\xi)\rvert < \infty, \qquad \forall N \in \mathbb{N}, a \in (0,1).
$$
Moreover, the topology of $\mathcal{U}_{W}(T^{F,G})$ is  generated by the system of norms $\{ \|\, \cdot \, \|_{N,a}\mid N \in \mathbb{N}, a \in (0,1)\}$.
\end{lemma}
Next, we discuss the role of the conditions \((N)\) and \((\epsilon_0)\).

\begin{proposition}\label{nuclear}
 Let \(F, G \in \mathcal{F}(\mathbb{R})\) and let  \(W\) be a weight function system that satisfies \((\alpha)\). The space \(\mathcal{K}_W(T^{F,G})\) is nuclear if and only if \(W\) satisfies \((N)\).
\end{proposition}
\begin{proof}
This can be shown in the same way as \cite[Theorems 5.1 and 5.3]{Nuclearity}, with obvious modifications; the details are left to the reader.
\end{proof}

The following result is the main theorem from  \cite{Surjectivity}.
\begin{theorem}  \cite[Theorem 3.1]{Surjectivity}\label{Summary surjectivity}  Let \(F, G \in \mathcal{F}(\mathbb{R})\)  and let \(W\) be a weight function system that satisfies \((\alpha)\) and  \((N)\). The following statements are equivalent:
       \begin{itemize} 
        \item  The Cauchy--Riemann operator
        \(\overline{\partial} \colon \mathcal{K}_W(T^{F,G}) \rightarrow \mathcal{K}_W(T^{F,G})\) is surjective.
        \item  The space \(\mathcal{U}_W(T^{F,G})\) is non-trivial. 
        \item \(W\) satisfies \((\epsilon)_0\).
        \end{itemize}
\end{theorem}
\begin{remark}\label{remark non-triviality}
Let \(W\) be a weight function system that satisfies \((\alpha)\) and  \((N)\). In addition to Theorem \ref{Summary surjectivity}, we remark that \(\mathcal{U}_W(\mathbb{C})\) is non-trivial if and only if  \(W\) satisfies \((\epsilon)_0\)  \cite[Remark 3.8]{Surjectivity}.
\end{remark}
Finally, we define the weighted spaces of vector-valued smooth functions we shall be interested in.\begin{definition}
Let \(F, G \in \mathcal{F}(\mathbb{R})\), let \(W =  (w_N)_{N \in \mathbb{N}}\) be a weight function system, and let \(E\) be a lcHs. We write  \(\mathcal{P}_E\)  for the family of continuous seminorms on $E$. We define $\mathcal{K}_{W}(T^{F,G};E)$ as the space consisting of all $\mathbf{f} \in C^{\infty}(T^{F,G};E)$ such that 
$$
p_{N,\alpha,a}(\mathbf{f})=\sup_{\xi \in \overline{T^{a F, a G}}} e^{w_N(\lvert \mathrm{Re \ }\xi\rvert)} p(\mathbf{f}^{(\alpha)}(\xi)) < \infty, \qquad \forall  p \in \mathcal{P}_E, N \in \mathbb{N}, \alpha \in \mathbb{N}^2, a \in (0,1).
$$
We endow $\mathcal{K}_{W}(T^{F,G};E)$ with the locally convex topology generated by the system of seminorms $\{ p_{N,\alpha,a} \mid p \in \mathcal{P}_E, N \in \mathbb{N}, \alpha \in \mathbb{N}^2, a \in (0,1)\}$. 
\end{definition}
We now provide an $\varepsilon$-product representation for the space  $\mathcal{K}_W(T^{F,G};E)$. To this end, we employ the following lemma from \cite{BFJ}.
\begin{lemma}\cite[Proposition 7]{BFJ} \label{lemma Bonet}
    Let \(Y\) be a Fréchet-Schwartz space, let \(X\subseteq Y_c^{\prime}\) be a dense subspace, and let \(E\) be a locally complete lcHs. For  a linear map \(T\colon X\rightarrow E\),  there exists a (unique) extension \(\widehat{T}\in L(Y_c^{\prime},E)\) of \(T\) if and only if  \(T^t(E^\prime)\subseteq Y\).
\end{lemma}

\begin{lemma}\label{epsilon-rep}
Let \(F,G \in \mathcal{F}(\mathbb{R})\),  let \(W=(w_N)_{N\in\mathbb{N}}\) be a weight function system, and let \(E\) be a locally complete lcHs. 
\begin{itemize}
\item[$(i)$] For all $f \in C^{\infty}(T^{F,G};E)$ it holds that
\begin{equation}\label{exchange derivative}
\langle x', \mathbf{f} \rangle^{(\alpha)} = \langle x', \mathbf{f}^{(\alpha)} \rangle, \qquad \forall \alpha \in \mathbb{N}^2, x' \in E'.
\end{equation}
\item[$(ii)$]  A function $\mathbf{f}: T^{F,G}\rightarrow E$ belongs to   $\mathcal{K}_W(T^{F,G};E)$ if and only if   $\langle x', \mathbf{f} \rangle = x' \circ \mathbf{f}:  T^{F,G}\rightarrow \mathbb{C}$ belongs to $\mathcal{K}_W(T^{F,G})$ for all $x' \in E^{\prime}$.
\item[$(iii)$] Let \(W\) satisfy \((\alpha)\) and \((N)\). The map 
\[ \begin{gathered}
\Phi\colon\mathcal{K}_W(T^{F,G};E) \to  \mathcal{K}_W(T^{F,G}) \varepsilon E = L(E'_c, \mathcal{K}_W(T^{F,G})), \,\mathbf{f} \mapsto (x' \mapsto \langle x', \mathbf{f} \rangle)
\end{gathered}
\]
is an isomorphism (of vector spaces).
\end{itemize}
\end{lemma}
\begin{proof}
Statement $(i)$ is clear, so we continue to show $(ii)$. As $E$ is locally complete, it holds that $\mathbf{f}: T^{F,G}\rightarrow E$ belongs to   $C^\infty(T^{F,G};E)$ if and only if   $\langle x', \mathbf{f} \rangle \in C^\infty(T^{F,G})$ for all $x' \in E^{\prime}$ \cite[p.~4]{BFJ}.
Observe that \(\mathbf{f}\in C^{\infty}(T^{F,G};E)\) belongs to \(\mathcal{K}_W(T^{F,G};E)\) if and only if \(\left\{e^{w_N(\lvert \mathrm{Re}z\rvert)} {\mathbf{f}}^{(\alpha)}(z) \mid z \in \overline{T^{aF,aG}}\right\}\subseteq E\) is bounded for all \(a\in (0,1)\), \(N\in\mathbb{N}\), and \(\alpha \in \mathbb{N}^2\). Hence, the result follows from $(i)$ and the fact that a set in $E$ is bounded if and only if it is weakly bounded (Mackey's theorem). 
Next, we prove \((iii)\). We first show that $\Phi$ is well-defined. Take an arbitrary \(\mathbf{f}\in \mathcal{K}_W(T^{F,G};E)\). Define the vector space \(X=\mathrm{span} \{\delta_z \mid z \in T^{F,G}\} \subseteq (\mathcal{K}_W(T^{F,G}))^{\prime}\) and let \(T\colon X \rightarrow E\) be the unique linear map with \(T(\delta_z)=\mathbf{f}(z)\) for all \(z\in T^{F,G}\). We claim  that \(T\) extends to an element \(\widehat{T}\in L((\mathcal{K}_W(T^{F,G}))^{\prime}_c,E)\). Then, \(\widehat{T}^t\in L(E^{\prime}_c,\mathcal{K}_W(T^{F,G}))\) \cite[p.~657]{K-Ultradistributions3}. As \(\widehat{T}^t(x^{\prime})=\langle x^{\prime},\mathbf{f}\rangle\) for all \(x^{\prime}\in E^\prime\), this would imply that \(\Phi\) is well-defined. We now show the claim by verifying the assumptions of Lemma \ref{lemma Bonet}. Proposition \ref{nuclear}  yields that the space \(\mathcal{K}_W(T^{F,G})\) is nuclear and thus Schwartz. Since the dual of \((\mathcal{K}_W(T^{F,G}))_c^{\prime}\) is \(\mathcal{K}_W(T^{F,G})\) \cite[p.~16]{S-TheorieDistValeurVect}, the Hahn-Banach theorem implies that \(X\) is dense in \((\mathcal{K}_W(T^{F,G}))_c^{\prime}\).  As \(T^t(x^{\prime})=\langle x^{\prime},\mathbf{f}\rangle\in\mathcal{K}_W(T^{F,G})\) for all \(x^{\prime}\in E^{\prime}\),  the inclusion \(T^t(E^{\prime})\subseteq \mathcal{K}_W(T^{F,G})\) follows from \((ii)\). Next, we show  that \(\Phi\) is a bijection. The map
\[
 C^{\infty}(T^{F,G};E)\rightarrow C^{\infty}(T^{F,G})\varepsilon E= L(E^{\prime}_c,C^{\infty}(T^{F,G}))\colon \mathbf{f}\mapsto (x^{\prime}\mapsto \langle x^{\prime},\mathbf{f}\rangle),
\]
 is an isomorphism of vector spaces \cite[p.~4]{BFJ}.  Hence, the result follows from $(ii)$.
\end{proof}
\subsection{Linear topological invariants} We start with the definition of  the condition $(\Omega)$ for Fr\'echet spaces \cite{Meise-Vogt,V-FuncExt1Frechet}.
\begin{definition}\label{Omega}
Let \(F\) be a Fréchet space with a  fundamental  non-decreasing sequence of seminorms \((\|\cdot\|_n)_{n\in\mathbb{N}}\). We say that \(F\) satisfies \((\Omega)\) if  
\[
\begin{gathered}
    \forall n \in \mathbb{N} \, \exists m \geq n \, \forall k \geq m \, \exists \theta \in (0,1), C > 0: \\
     \|x'\|^*_m\leq C(\|x'\|^*_n)^{1-\theta} (\|x'\|^*_k)^{\theta}, \qquad \forall x' \in F'.
\end{gathered}
\]
\end{definition}
    Next, we introduce the condition \((A)\)  for lcHs with a fundamental sequence of bounded sets \cite{V-VektorDistrRandHolomorpherFunk}. \begin{definition}\label{DEFA}
Let  $E$ be a lcHs with a fundamental non-decreasing sequence  of absolutely convex bounded sets $(B_{n})_{n \in \mathbb{N}}$. We say that $E$ satisfies  \((A)\) if
			\[
			\begin{gathered}
				\exists n \in \mathbb{N} \, \forall m \geq n,  s > 0 \, \exists k \geq m, C > 0 :  \\
				B_{m} \subseteq  r B_{n} + \frac{C}{r^{s}} B_{k}, \qquad \forall r > 0.
				\end{gathered}
			\]
	\end{definition}
	
	\begin{remark} \label{remark-DN}
	The condition \((A)\)  is closely related to the condition $(DN)$ for Fr\'echet spaces  \cite{Meise-Vogt,V-FuncExt1Frechet}. A Fr\'{e}chet space $F$ with a fundamental non-decreasing sequence of seminorms $(\| \,\cdot \,\|_{n})_{n \in \mathbb{N}}$ is said to satisfy $(DN)$ if 
							\[
	\begin{gathered}
						\exists n \in \mathbb{N} \, \forall m \geq n,  \theta \in (0, 1) \, \exists k \geq m, C > 0  : \\  \|x\|_{m} \leq C \|x\|^{\theta}_{n} \|x\|^{1 - \theta}_{k}, \qquad \forall x \in F.
					\end{gathered}
					\]
			A Fr\'echet space $F$ satisfies $(DN)$ if and only if $F^{\prime}$ satisfies $(A)$ (cf.\ \cite[Proposition 2.2]{B-D-V-InterpolVVRealAnalFunc}).
	\end{remark}

We will use the following result about the surjectivity of $\varepsilon$-tensorized mappings from \cite{D-N}, which stems from the work of Vogt \cite{V-TensorFundDFRaumFortsetz}. 

\begin{theorem}\label{surjtensor}\cite[Theorem 6.3]{D-N}
   Let \(F\) and \(G\) be two nuclear Fréchet spaces and let \(S\colon F \rightarrow G\) be a surjective continuous linear map. Let $E$ be a locally complete lcHs with a fundamental sequence of bounded sets.
If \(\mathrm{ker} S\) satisfies \((\Omega)\) and $E$ satisfies  \((A)\), then the  tensorized map 
   \[
S \varepsilon \mathrm{id}_E \colon F  \varepsilon E \rightarrow G\varepsilon E 
   \]
is surjective.
\end{theorem}

\begin{remark}
If additionally $E$ is  Montel, Theorem \ref{surjtensor} is a consequence of the classical ($DN$)-($\Omega$) splitting theorem of Vogt and Wagner \cite{V-W-CharQuotientsMartineau, V-FuncExt1Frechet} (see the introduction of  \cite{D-N} for more information).
\end{remark}
 
 \section{Main results}\label{sect-main results}
 In this section, we state our main results and discuss some examples. 

 \subsection{The condition $(\Omega)$ for    \(\mathcal{U}_W(T^{F,G})\)} The principal result of this article is a characterization of the condition \((\Omega)\) for the space \(\mathcal{U}_W(T^{F,G})\) in terms of a concrete condition on \(W\). In order to state it, we need the following definition.
 \begin{definition}
Let \(W=(w_N)_{N\in\mathbb{N}}\) be a weight function system. We say that $W$ satisfies the condition $(\omega)$ if
\[
\begin{gathered}
\forall N \in \mathbb{N} \, \exists M \geq N \,\forall  K \geq M \, \exists \theta \in (0,1), C > 1: \\
(1-\theta) w_N(t)+\theta w_K(t)\leq w_M(t) + \log C, \qquad \forall t \geq 0.
\end{gathered}
\]
 \end{definition}
 \begin{theorem}\label{theorem Omega}
    Let \(F,G \in \mathcal{F}(\mathbb{R})\) and let \(W\) be a weight function system satisfying \((\alpha)\) and \((N)\). Suppose that \(\mathcal{U}_W(T^{F,G}) \not= \{0\}\). Then, \(\mathcal{U}_W(T^{F,G})\) satisfies $(\Omega)$ if and only if \(W\) satisfies \((\omega)\).
\end{theorem}
The proof of Theorem \ref{theorem Omega} will be given in Section \ref{proof theorem omega}. 

We now discuss Theorem \ref{theorem Omega} for weight function systems generated by a single weight function \(w\). To this end, we consider the following conditions on \(w\):
\begin{enumerate}
    \item [\((\alpha)\)] \(\exists C > 0, A > 1 \, \forall t \geq 0: w(2t) \leq Cw(t) + \log A\) .\\
    \item[\((\delta)\)]
     \(\exists C>0,A > 1\, \forall t \geq 0: 2w(t) \leq w(Ct) + \log A\). \\
    \item [\((N)_o\)] \(\exists C > 0: \displaystyle\int_0^{\infty} e^{-Cw(t)}\mathrm{d}t < \infty\). \\
    \item [\((\epsilon)_0\)] \(\forall \mu>0: \displaystyle\int_0^{\infty}e^{-\mu t}w(t)\mathrm{d}t<\infty\).
\end{enumerate}
Following \cite{Surjectivity},  we define the weight function systems \(\tilde{W}_w=(Nw)_{N\in\mathbb{N}}\) and  \(W_w=(w(Nt))_{N\in\mathbb{N}}\).
\begin{lemma}\label{N outside}
 Let \(w\) be a weight function. Then,
 \begin{enumerate}[label=(\roman*)]
     \item \(\tilde{W}_w\) satisfies \((\alpha)\) if and only if \(w\) satisfies \((\alpha)\).
     \item \(\tilde{W}_w\) satisfies \((N)\) if and only if \(w\) satisfies \((N)_o\).
     \item \(\tilde{W}_w\) satisfies \((\epsilon)_0\) if and only if \(w\) satisfies \((\epsilon)_0\).
     \item \(\tilde{W}_w\) always satisfies \((\omega)\).
 \end{enumerate}
\end{lemma}
\begin{proof}
    Statements \((i)\)--\((iii)\) are obvious, so we only show \((iv)\). To this end, let \(N,M,K\in\mathbb{N}\) with \(N<M<K\) be arbitrary.  For \(0<\theta<(M-N)/(K-N)\) it holds that \((1-\theta)N+\theta K< M\). This implies that \(\tilde{W}_w\) satisfies \((\omega)\).
\end{proof}
Hence, Theorem \ref{theorem Omega} yields the following result for weight function systems of type \(\tilde{W}_w\).
\begin{corollary}\label{omega for N outside}
Let \(F,G\in\mathcal{F}(\mathbb{R})\) and let \(w\) be a weight function that satisfies \((\alpha)\) and \((N)_o\). Then, \(\mathcal{U}_{\tilde{W}_w}(T^{F,G})\) satisfies \((\Omega)\). 
\end{corollary}
Next, we consider weight function systems of type \(W_w\).
\begin{lemma}\label{N inside}
    Let \(w\) be a weight function. Then,
    \begin{enumerate}[label=(\roman*)]
        \item \(W_w\) always satisfies \((\alpha)\).
        \item \(W_w\) satisfies \((\epsilon)_0\) if and only if \(w\) satisfies \((\epsilon)_0\). 
        \item  If \(w\) satisfies \((\alpha)\), then \(w\) satisfies \((\epsilon)_0\).
        \item  If \(w\) satisfies \((\delta)\), then \(W_w\) satisfies \((N)\). 
        \item If \(W_w\) satisfies \((\omega)\), then \(w\) satisfies \((\alpha)\).
        \item Let \(w\) satisfy \((\delta)\). Then, \(W_w\) satisfies \((\omega)\) if and only if \(w\) satisfies \((\alpha)\). 
        \end{enumerate}
    Moreover, if \(w\) satisfies \((\alpha)\) and \((\delta)\), then \(\mathcal{K}_{\tilde{W}_w}(T^{F,G}) = \mathcal{K}_{W_w}(T^{F,G})\) as Fr\'echet spaces for all \(F,G\in\mathcal{F}(\mathbb{R})\).
\end{lemma}
\begin{proof}
    Statement \((i)\) and \((ii)\) are obvious, while statements \((iii)\) and \((iv)\) are shown in \cite[Lemmas 3.2 and 3.5]{Surjectivity}. We now show \((v)\). Assume that \(W_w\) satisfies \((\omega)\). There are \(M\in\mathbb{N}\), \(\theta\in(0,1)\), and \(C>1\) such that
    \[
\theta w(2Mt)\leq (1-\theta)w(t)+\theta w(2Mt)\leq w(Mt) +  \log C , \qquad \forall t \geq 0.
    \]
    This implies 
    \( w(2s)\leq w(s)/\theta+\log C/\theta\) for all $s \geq 0$. Next, we prove \((vi)\). By \((v)\), it suffices to show that \(W_w\) satisfies \((\omega)\) if \(w\) satisfies \((\alpha)\) and \((\delta)\). The latter assumptions imply that for all \(N\in\mathbb{N}\) there are \(M\in\mathbb{N}\) and \(A>1\) such that 
    \begin{equation}\label{N-inside is N-outside}
 w(Nt)\leq Mw(t)+\log A \text{ and } Nw(t)\leq w(Mt)+\log A, \qquad \forall t \geq 0.
    \end{equation} 
    Hence, \(W_w\) satisfies \((\omega)\) since \(\tilde{W}_w\) does so  (Lemma \ref{N outside}$(iv)$).
Finally, the last statement follows from  \eqref{N-inside is N-outside}. 
\end{proof}
Theorem \ref{theorem Omega} implies the following result for weight function systems of type \(W_w\). 
\begin{corollary}\label{omega for N inside}
    Let \(F,G\in\mathcal{F}(\mathbb{R})\) and let \(w\) be a weight function satisfying \((\delta)\). Suppose that \(\mathcal{U}_{W_w}(T^{F,G})\not=\{0\}\). Then, \(\mathcal{U}_{W_w}(T^{F,G})\) satisfies \((\Omega)\) if and only if \(w\) satisfies \((\alpha)\). 
\end{corollary}
\begin{proof}
    The result follows by combining Theorem \ref{theorem Omega} and  Lemma \ref{N inside}.
\end{proof}
Finally, we give some examples of weight function systems \(W\) for which \(\mathcal{U}_W(T^{F,G})\) does/does not have \((\Omega)\).
\begin{example}\label{example weight function omega N-outside}
 Let \(a > 0\) and \(b \in \mathbb{R}\). Let \(w\) be a weight function such that $w(t)= t^a (\log(e+t))^b$ for \(t\) sufficiently large. Then, $w$ always satisfies \((\alpha)\) and \((\delta)\).  Hence, \(\mathcal{U}_{W_w}(T^{F,G})=\mathcal{U}_{\tilde{W}_w}(T^{F,G})\) satisfies \((\Omega)\) for all \(F,G\in\mathcal{F}(\mathbb{R})\).
\end{example}
\begin{example}
    Let \(0 < a < 1\), \(b\in \mathbb{R}\) or \(a=0\), \(b>1\).  Let $w$ be a weight function such that \(\displaystyle w(t)=e^{t^a(\log(e+t))^b}\) for \(t\) sufficiently large. Then, \(w\)  satisfies \((\delta)\) and $(\epsilon)_0$, but not \((\alpha)\). Hence, for all \(F,G\in\mathcal{F}(\mathbb{R})\), \(\mathcal{U}_{W_w}(T^{F,G})\not=\{0\}\) does not satisfy \((\Omega)\). We do not know whether \(\mathcal{U}_{\tilde{W}_w}(T^{F,G})\) satisfies \((\Omega)\). For $T^{F,G} = T^h$, $h >0$, a horizontal strip, this follows from \cite[Proposition 3.1]{Sequence space representations}. 
\end{example}
 \subsection{Surjectivity of the Cauchy--Riemann operator on   \(\mathcal{K}_W(T^{F,G};E)\)} 
We have the following consequence of Theorem \ref{theorem Omega}.
 
 \begin{theorem}\label{Vector-valued surjectivity}
Let \(F,G\in\mathcal{F}(\mathbb{R})\) and let \(W\) be a weight function system satisfying  \((\alpha)\), \((N)\), and \((\omega)\). Suppose that the Cauchy--Riemann operator
   \begin{equation}
   \label{deltabar}
    \overline{\partial}\colon \mathcal{K}_W(T^{F,G})\rightarrow\mathcal{K}_W(T^{F,G})
\end{equation}
    is surjective. Let $E$ be a locally complete lcHs with a fundamental sequence of bounded sets. If $E$ satisfies  \((A)\), then the  $E$-valued Cauchy--Riemann operator
   \begin{equation}
   \label{vvdeltabar}
    \overline{\partial}\colon \mathcal{K}_W(T^{F,G};E)\rightarrow\mathcal{K}_W(T^{F,G};E)
    \end{equation}
is surjective as well.
\end{theorem}
\begin{proof}
Note that $\mathcal{U}_W(T^{F,G})$ is the kernel of the map  \(\overline{\partial}\colon \mathcal{K}_W(T^{F,G})\rightarrow\mathcal{K}_W(T^{F,G})\).  By Proposition \ref{nuclear} and Theorem  \ref{theorem Omega}, $\mathcal{U}_W(T^{F,G})$ is a nuclear Fr\'echet space that satisfies $(\Omega)$. Hence, Theorem \ref{surjtensor} yields that the map \(\overline{\partial} \varepsilon \mathrm{id}_E\) is surjective on \(\mathcal{K}_W(T^{F,G})\varepsilon E\). Lemma \ref{epsilon-rep}$(i)$ implies that,  under the isomorphism \(\Phi: \mathcal{K}_W(T^{F,G}; E)  \to \mathcal{K}_W(T^{F,G})\varepsilon E\) from Lemma \ref{epsilon-rep}$(iii)$,  the map  \(\overline{\partial} \varepsilon \mathrm{id}_E\) is  given by \eqref{vvdeltabar}. This shows the result.
\end{proof}

\begin{remark}\label{Frechet}
$(i)$ Let \(F,G\in\mathcal{F}(\mathbb{R})\) and let \(W\) be a weight function system satisfying  \((\alpha)\) and \((N)\). Recall from Theorem \ref{Summary surjectivity} that the map \eqref{deltabar} is surjective if and only if \(W\) satisfies $(\epsilon_0)$. \\
$(ii)$ Theorem \ref{Vector-valued surjectivity} is particularly applicable to weight function systems of the form $\tilde{W}_w$, with $w$ a weight function satisfying $(\alpha)$ and  \((N)_o\)  (Lemma \ref{N outside}). \\
$(iii)$ Let \(F,G\in\mathcal{F}(\mathbb{R})\) and let \(W\) be a weight function system satisfying  \((\alpha)\) and \((N)\). Suppose that the map \eqref{deltabar} is surjective. Then, for any Fr\'echet space $E$,  the $E$-valued Cauchy-Riemann operator \eqref{vvdeltabar} is surjective: As in the proof of Theorem \ref{Vector-valued surjectivity}, it suffices to show  the map \(\overline{\partial} \varepsilon \mathrm{id}_E\) is surjective on \(\mathcal{K}_W(T^{F,G})\varepsilon E\). Since \(\mathcal{K}_W(T^{F,G})\) is complete and nuclear (Proposition \ref{nuclear}), it holds that  \(\mathcal{K}_W(T^{F,G})\varepsilon E\) =  \(\mathcal{K}_W(T^{F,G}) \widehat{\otimes}_\pi E\) (cf.\ \cite[Proposition 1.4]{K-Ultradistributions3}).  The result now follow from the fact that the $\pi$-completed tensor product  of two surjective continuous linear maps between Fr\'echet spaces is again surjective \cite[Proposition 43.9]{Treves}.
\end{remark}

We have the following converse to Theorem  \ref{Vector-valued surjectivity}. 
\begin{theorem}\label{Converse}
 Let \(F,G\in\mathcal{F}(\mathbb{R})\) and let \(W\) be a weight function system. Let $E$ be a locally complete lcHs with a fundamental sequence of bounded sets. If the  $E$-valued Cauchy--Riemann operator  \eqref{vvdeltabar}
 is surjective, then $E$ satisfies $(A)$. 
\end{theorem}
The proof of Theorem \ref{Converse} will be given in Section \ref{proof theorem omega}. 

We end this subsection with two consequences of Theorem \ref{Converse}, which can be interpreted as parameter dependence results for the Cauchy-Riemann operator.
Given an open set set $X \subseteq \mathbb{R}^d$, we denote by $\mathcal{E}'(X)$ the strong dual of the Fr\'echet space $C^{\infty}(X)$. Note that $\mathcal{E}'(X)$ is the space of distributions having compact support contained in $X$.

\begin{corollary}
    Let \(F,G\in\mathcal{F}(\mathbb{R})\),  let \(W\) be a weight function system, and  let \(X\subseteq \mathbb{R}^d\) be open. The map 
    \[
    \overline{\partial}\colon \mathcal{K}_W(T^{F,G}; \mathcal{E}'(X) )\rightarrow \mathcal{K}_W(T^{F,G};\mathcal{E}'(X))
    \] is not surjective.
\end{corollary}
\begin{proof}
In view of Theorem \ref{Converse}, it suffices to show that $\mathcal{E}'(X)$ does not satisfy $(A)$. By Remark \ref{remark-DN}, this follows from the fact that \(C^{\infty}(X)\) does not satisfy \((DN)\) (this space does not even have a  continuous norm). 
\end{proof}
We write $\mathcal{A}(\mathbb{R})$ for the space of real analytic functions on $\mathbb{R}$, endowed with its natural locally convex topology. We refer to \cite[Chapter 29]{Meise-Vogt} for the definition of power series spaces.
\begin{corollary}\label{real analytic}
Let \(F,G\in\mathcal{F}(\mathbb{R})\) and  let \(W\) be a weight function system. The map 
    \begin{equation}
    \label{realanalyticvv}
    \overline{\partial}\colon \mathcal{K}_W(T^{F,G}; \mathcal{A}(\mathbb{R}) )\rightarrow \mathcal{K}_W(T^{F,G};\mathcal{A}(\mathbb{R}))
   \end{equation}
   is not surjective.
\end{corollary}
\begin{proof}
By \cite[Corollary 4.3]{Langenbruch2011}, the space $\mathcal{A}_{\operatorname{per}}(\mathbb{R})$ of $2\pi$-periodic real analytic functions on $\mathbb{R}$ is complemented in $\mathcal{A}(\mathbb{R})$. If the map \eqref{realanalyticvv} would be  surjective, then    $\overline{\partial}$ would also be surjective on $\mathcal{K}_W(T^{F,G};\mathcal{A}_{\operatorname{per}}(\mathbb{R}))$. We now show that the latter is not true. By Theorem \ref{Converse}, it suffices to show that $\mathcal{A}_{\operatorname{per}}(\mathbb{R})$ does not satisfy $(A)$. Note that $\mathcal{A}_{\operatorname{per}}(\mathbb{R})$ is isomorphic to the strong dual of a power series space $\Lambda_0(j)$ of finite type via Fourier coefficients. The result now follows from  Remark \ref{remark-DN} and the fact  that power series spaces of finite type do not satisfy $(DN)$ \cite[Proof of Proposition 29.3]{Meise-Vogt}.
\end{proof}
\section{Proof of Theorem \ref{theorem Omega}}\label{proof theorem omega}
This section is dedicated to the proof of Theorem \ref{theorem Omega}. We split it into two parts.
\subsection{Sufficiency of the condition $(\omega)$} Our goal is to show  the condition $(\Omega)$ --a log-convex inequality for the dual norms-- for the Fr\'echet space  $\mathcal{U}_W(T^{F,G})$. As  explained in the introduction, our proof method consists of two main steps. First, we represent the dual of $\mathcal{U}_W(T^{F,G})$ by a weighted space of  holomorphic functions.  To this end, we associate to each $f \in (\mathcal{U}_W(T^{F,G}))'$  a holomorphic function defined on an open set containing $\mathbb{C} \backslash T^{F,G}$, called \emph{the weighted Cauchy transform} of \(f\), and show that $f$ may be seen as the boundary value of this function. This allows us to express the dual norms of $f$ in terms of weighted supremum norms of its weighted Cauchy transform.  Next, we show  a log-convex inequality for the weighted spaces of holomorphic functions representing  $(\mathcal{U}_W(T^{F,G}))'$. The condition $(\Omega)$ for $\mathcal{U}_W(T^{F,G})$ then easily follows by combining these two parts.
We start with some preparatory lemmas. Consider the following three variants of the conditions $(\alpha)$, $(N)$, and $(\omega)$ for a weight function system \(W = (w_N)_{N\in \mathbb{N}}\):
\begin{enumerate}
    \item [\((\tilde{\alpha})\)] \(\forall N \in \mathbb{N} \, \exists A > 1 \, \forall t\geq 0 : w_N(2t) \leq w_{N+1}(t) + \log A\).  
      \\
  \item [ \((\tilde{N})\)] \(\forall N \in \mathbb{N} : \displaystyle{\int_{0}^{\infty} e^{w_N(t) - w_{N+1}(t)} \mathrm{d}t< \infty}\).
  \item[ \((\tilde{\omega})\)] \( \forall N \in \mathbb{N}, K \geq N+1 \, \exists \theta \in (0,1), C > 1: \)
  $$
(1-\theta) w_N(t)+\theta w_K(t)\leq w_{N+1}(t) + \log C, \qquad \forall t \geq 0. 
$$
  \end{enumerate}
\begin{lemma}\label{reduction}
Let  \(W = (w_N)_{N\in \mathbb{N}}\) be a weight function system  satisfying $(\alpha)$ and/or $(N)$ and/or $(\omega)$. Then, there is a weight function system \(\tilde{W}\) satisfying  $(\tilde{\alpha})$ and/or $(\tilde{N})$ and/or $(\tilde{\omega})$ such that $\mathcal{U}_W(T^{F,G}) = \mathcal{U}_{\tilde{W}}(T^{F,G})$ as Fr\'echet spaces for all $F, G \in \mathcal{F}(\mathbb{R})$.
\end{lemma}
\begin{proof}
Since  \(W\)  satisfies $(\alpha)$ and/or $(N)$ and/or $(\omega)$, there is an increasing sequence $(N_k)_{k \in \mathbb{N}}$ of natural numbers such that  the weight function system  \(\tilde{W}  = (w_{N_k})_{k\in \mathbb{N}}\) satisfies  $(\tilde{\alpha})$ and/or $(\tilde{N})$ and/or $(\tilde{\omega})$. It is clear that  also $\mathcal{U}_W(T^{F,G}) = \mathcal{U}_{\tilde{W}}(T^{F,G})$ as Fr\'echet spaces for all $F, G \in \mathcal{F}(\mathbb{R})$.
\end{proof}
Lemma \ref{reduction} will allow us to work with the conditions $(\tilde{\alpha})$, $(\tilde{N})$, and $(\tilde{\omega})$ instead of $(\alpha)$, $(N)$, and $(\omega)$. This will simplify the notation a bit. We also need the following two lemmas.
\begin{lemma}\label{wsuba}
Let  \(W = (w_N)_{N\in \mathbb{N}}\) be a weight function system  satisfying $(\tilde{\alpha})$. For all $N \in \mathbb{N}$ there is $A_N > 1$ such that 
$$
w_N(t+s) \leq w_{N+1}(t) + w_{N+1}(s) + \log A_N, \qquad \forall t,s \geq 0.
$$
\end{lemma}
\begin{proof}
Let $N \in \mathbb{N}$ be arbitrary. There is $A_N >1$ such that  $w_N(2t) \leq w_{N+1}(t) + \log A_N$ for all $t \geq 0$. Hence,  for all $t,s \geq 0$ 
$$
w_N(t+s)  \leq w_N(2 \max\{t,s\}) \leq  w_{N+1}(\max\{t,s\}) + \log A_N \leq  w_{N+1}(t) + w_{N+1}(s) + \log A_N.
$$
\end{proof}
\begin{lemma}\label{N for sequences}
    Let \(W=(w_N)_{N\in\mathbb{N}}\) be a weight function system that satisfies \((\tilde{\alpha})\) and \((\tilde{N})\). Then, $w_N(j)-w_{N+2}(j) \to - \infty$ as $j \to \infty$.
\end{lemma}
\begin{proof}
Let $A_N >1$, $N \in \mathbb{N}$,  be as in Lemma \ref{wsuba}. Then, for all \(N\in\mathbb{N}\) 
\begin{align*}
\infty >    \int_{0}^{\infty} e^{w_N(t)-w_{N+1}(t)}\mathrm{d}t = \sum_{j=0}^{\infty}\int_{j}^{j+1}e^{w_N(t)-w_{N+1}(t)}\mathrm{d}t &\geq \sum_{j=0}^{\infty}e^{w_N(j)-w_{N+1}(j+1)} \\
&\geq \frac{1}{A_{N+1}e^{w_{N+2}(1)}} \sum_{j=0}^{\infty} e^{w_N(j)-w_{N+2}(j)}.
\end{align*}
This implies the result.
\end{proof}

\subsubsection{The weighted Cauchy transform}\label{WCT} Throughout this paragraph, we fix \(F,G \in \mathcal{F}(\mathbb{R})\) and \(h>0\) such that $\overline{T^{F,G}} \subseteq T^h$. Furthermore, we  fix a weight function system \(W= (w_N)_{N\in\mathbb{N}}\)  that satisfies \((\tilde{\alpha})\) and  \((\tilde{N})\).  We need the following Banach spaces of holomorphic functions.

\begin{definition}
For $N \in \mathbb{N}$ we define the Banach spaces
 \[
 \mathcal{H}_N(T^h\setminus \overline{T^{F,G}}) = \left\{U\in\mathcal{H}(T^h\setminus \overline{T^{F,G}}) \mid |U|_{N,F,G} = \sup_{z\in T^h\setminus \overline{T^{F,G}}} e^{-w_N(\lvert \mathrm{Re}z\rvert)} \lvert U(z)\rvert < \infty\right\} 
 \]
 and
 \[
 \mathcal{U}_N(\overline{T^{F,G}}) = \left\{\varphi\in \mathcal{H}(T^{F,G})\cap C(\overline{T^{F,G}}) \mid \|\varphi\|_{N,F,G}= \sup_{z\in\overline{T^{F,G}}} e^{w_N(\lvert \mathrm{Re} z\rvert)} 
\lvert \varphi(z)\rvert < \infty\right\}.
 \]
 \end{definition}
 Let \(\Phi, \Psi \in \mathcal{F}(\mathbb{R})\) with \(\Phi<F\) and \(\Psi<G\), and let \(N\in\mathbb{N}\). 
We define $(\mathcal{U}_W(T^{F,G}))^{\prime}_{N,\Phi,\Psi}$ as the Banach space consisting of all $f \in (\mathcal{U}_W(T^{F,G}))^{\prime}$ such that
$$
\|f\|^*_{N,\Phi,\Psi} = \sup \{ | \langle f, \varphi \rangle| \mid {\varphi \in \mathcal{U}_W(T^{F,G}), \|\varphi\|_{N,\Phi,\Psi} \leq 1}\} <\infty.
 $$
By Remark \ref{stupid-remark}$(i)$, it holds that  for every $f  \in (\mathcal{U}_W(T^{F,G}))^{\prime}$ there are $\Phi,\Psi\in\mathcal{F}(\mathbb{R})$ with $\Phi<F$ and $\Psi<G$, and $N\in \mathbb{N}$ such that $f \in (\mathcal{U}_W(T^{F,G}))^{\prime}_{N,\Phi,\Psi}$.

Next, we define the weighted Cauchy transform.
\begin{definition}\label{weighted Cauchy transform} 
Let \(P \in \mathcal{U}_W(\mathbb{C})\). We define \emph{the weighted Cauchy transform of \(f\in (\mathcal{U}_W(T^{F,G}))^{\prime}\) (with respect to $P$)} as
    \[
    \mathrm{C}_P(f)(\xi) = \frac{1}{2\pi i}\left \langle f(z), \frac{P(\xi-z)}{\xi-z} \right\rangle, \qquad \xi \in \mathbb{C}\setminus \overline{T^{F,G}}.
    \]
\end{definition}
Since \(W\) satisfies \((\alpha)\), the space \(\mathcal{U}_W(\mathbb{C})\) is translation-invariant. Hence, the function $z \mapsto  P(\xi-z)/(\xi-z)$ belongs to $\mathcal{U}_W(T^{F,G})$ for all  $\xi \in \mathbb{C}\setminus \overline{T^{F,G}}$. This implies that   $\mathrm{C}_P(f)$ is well-defined and holomorphic on \(\mathbb{C}\setminus \overline{T^{F,G}}\) for all \(f\in (\mathcal{U}_W(T^{F,G}))^{\prime}\). In fact, we can say more:
\begin{lemma}\label{unique analytic extension}
  Let \(P \in \mathcal{U}_W(\mathbb{C})\), let \(\Phi,\Psi \in \mathcal{F}(\mathbb{R})\) with \(\Phi<F\) and \(\Psi<G\). Take \(f\in (\mathcal{U}_W(T^{F,G}))^{\prime}_{N, \Phi,\Psi}\). Then, \(\mathrm{C}_P(f)\) extends uniquely to a holomorphic function on \(\mathbb{C}\setminus \overline{T^{\Phi,\Psi}}\). 
\end{lemma}
\begin{proof} 
By the Hahn-Banach theorem,  there exists \(\tilde{f} \in  (\mathcal{U}_N(\overline{T^{\Phi,\Psi}}))'\) such that \(\tilde{f}_{\mid\mathcal{U}_W(T^{F,G})} = f\). The function 
$$
\displaystyle U(\xi) = \frac{1}{2\pi i}\left\langle \tilde{f}(z), \frac{P(\xi-z)}{\xi-z}\right\rangle, \qquad  \xi \in \mathbb{C}\setminus \overline{T^{\Phi,\Psi}},
$$
 is holomorphic on \(\mathbb{C}\setminus \overline{T^{\Phi,\Psi}}\) and coincides with \(\mathrm{C}_P(f)\) on \(\mathbb{C}\setminus \overline{T^{F,G}}\). The uniqueness of the  extension $U$ follows from the uniqueness principle for holomorphic functions.
\end{proof}
Let \(f\in (\mathcal{U}_W(T^{F,G}))^{\prime}_{N, \Phi,\psi}\). From now on, we will view \(\mathrm{C}_P(f)\) as a holomorphic function on \(\mathbb{C}\setminus \overline{T^{\Phi,\Psi}}\).
\begin{proposition}\label{continuity weighted Cauchy transform}
   Let \(P \in \mathcal{U}_W(\mathbb{C})\), let \(\Phi,\Psi,\tilde{\Phi},\tilde{\Psi} \in \mathcal{F}(\mathbb{R})\) with \(\Phi<\tilde{\Phi}<F\) and \(\Psi<\tilde{\Psi}<G\), and let $N \in \mathbb{N}$. The map 
\[
\mathrm{C}_P: (\mathcal{U}_W(T^{F,G}))^{\prime}_{N,\Phi,\Psi} \rightarrow \mathcal{H}_{N+1}\left(T^h\setminus \overline{T^{\tilde{\Phi},\tilde{\Psi}}}\right),\,  f\mapsto \mathrm{C}_P(f)
\]
is well-defined and continuous.
\end{proposition}
\begin{proof} Let $f \in  (\mathcal{U}_W(T^{F,G}))^{\prime}_{N,\Phi,\Psi}$ be arbitrary. By the Hahn-Banach theorem,  there exists \(\tilde{f} \in  (\mathcal{U}_N(\overline{T^{\Phi,\Psi}}))'\) such that \(\tilde{f}_{\mid\mathcal{U}_W(T^{F,G})} = f\) and 
\begin{equation}
\label{HB}
\|f\|^*_{N,\Phi,\Psi} = \sup \{ | \langle \tilde{f}, \psi \rangle| \mid {\psi \in \mathcal{U}_N(\overline{T^{\Phi,\Psi}}), \|\psi\|_{N,\Phi,\Psi} \leq 1}\}.
\end{equation}
From the proof of Lemma \ref{unique analytic extension}, it  follows that
$$
\mathrm{C}_P(f)(\xi) = \frac{1}{2\pi i}\left\langle \tilde{f}(z), \frac{P(\xi-z)}{\xi-z}\right\rangle, \qquad  \xi \in \mathbb{C}\setminus \overline{T^{\Phi,\Psi}}.
$$
Let $A_N >1$, $N \in \mathbb{N}$,  be as in Lemma \ref{wsuba}.  Let \(\tilde{h}>0\)  be such that \(T^h-\overline{T^{\Phi,\Psi}}\subseteq T^{\tilde{h}}\) and let \(\varepsilon > 0\) be such that \(\overline{T^{\Phi,\Psi}}+\overline{B}(0,\varepsilon)\subseteq T^{\tilde{\Phi},\tilde{\Psi}}\). By \eqref{HB}, we find that for all \(\xi \in T^h\setminus \overline{T^{\tilde{\Phi},\tilde{\Psi}}}\)
    \begin{align*}
       \lvert \mathrm{C}_P(f)(\xi)\rvert &=\frac{1}{2\pi} \left\lvert\left\langle\tilde{f}(z), \frac{P(\xi-z)}{\xi-z}\right\rangle\right\rvert\\
        &\leq \frac{1}{2\pi} \|f\|^{*}_{N,\Phi,\Psi} \sup_{z\in \overline{T^{\Phi,\Psi}}} e^{w_N(\lvert \mathrm{Re}z\rvert)}\frac{\lvert P(\xi-z)\rvert}{\lvert \xi-z\rvert} \\
        &\leq \frac{A_N \|P\|_{N+1,\tilde{h},\tilde{h}}}{2\pi \varepsilon} \|f\|^{*}_{N,\Phi,\Psi} e^{w_{N+1}(\lvert \mathrm{Re}\xi\rvert)}.
    \end{align*}   
   This implies that \(\mathrm{C}_P(f)\in\mathcal{H}_{N+1}(T^h\setminus \overline{T^{\tilde{\Phi},\tilde{\Psi}}})\) and the desired continuity of \(\mathrm{C}_P\).
\end{proof}
We now discuss boundary values of holomorphic functions in $(\mathcal{U}_W(T^{F,G}))'$. Recall from Definition \ref{gstrips} that  \(\mathcal{F}_1(\mathbb{R})\) consists of all \(F \in \mathcal{F}(\mathbb{R}) \cap C^1(\mathbb{R})\) with  $\sup_{t\in \mathbb{R}} |F'(t)| < \infty$. For \(\Phi\in \mathcal{F}_1(\mathbb{R})\)  we define the curves
    \[
    \Gamma^\pm_{\Phi}=\{t\pm i\Phi(t) \mid t\in\mathbb{R}\} 
    \]
    and orient \(\Gamma^+_{\Phi}\) from \enquote{right to left} and \(\Gamma^-_{\Phi}\) from \enquote{left to right}.
\begin{proposition}\label{bvbefore}
 Let \(\Phi,\Psi\in \mathcal{F}(\mathbb{R})\) with \(\Phi<F\) and \(\Psi<G\), and let  \(N\in\mathbb{N}\). Take  \(U\in \mathcal{H}_N(T^h\setminus \overline{T^{\Phi,\Psi}})\).
    \begin{enumerate}
        \item[$(i)$] Let \(\tilde{\Phi},\tilde{\Psi}\in\mathcal{F}_1(\mathbb{R})\) with \(\Phi<\tilde{\Phi}<F\) and \(\Psi<\tilde{\Psi}<G\). Define 
        \[
\langle \mathrm{bv}_{\tilde{\Phi},\tilde{\Psi}}(U), \varphi\rangle = \int_{\Gamma^+_{\tilde{\Phi}}}U(z)\varphi(z)\mathrm{d}z+\int_{\Gamma^-_{\tilde{\Psi}}}U(z)\varphi(z)\mathrm{d}z, \qquad \varphi \in \mathcal{U}_W(T^{F,G}).
\]
Then, $\mathrm{bv}_{\tilde{\Phi},\tilde{\Psi}}(U) \in (\mathcal{U}_W(T^{F,G}))'$ and 
$$
\|\mathrm{bv}(U)\|^*_{N+1,\tilde{\Phi},\tilde{\Psi}} \leq C \ \sup_{\xi\in\Gamma^+_{{\tilde{\Phi}}} \cup \Gamma^-_{{\tilde{\Psi}}} } e^{-w_N(\lvert\mathrm{Re}\xi\rvert)}\lvert U(\xi)\rvert
$$ with 
$$
\displaystyle C = (2+ \sup_{t\in\mathbb{R}}\lvert \tilde{\Phi}^{\prime}(t)\rvert +\sup_{t\in\mathbb{R}}\lvert {\tilde{\Psi}}^{\prime}(t)\rvert ) \int_0^{\infty}e^{w_N(t)-w_{N+1}(t)}\mathrm{d}t < \infty.
$$
  \item[$(ii)$] Let \(\tilde{\Phi}_k,\tilde{\Psi}_k\in\mathcal{F}_1(\mathbb{R})\) with \(\Phi<\tilde{\Phi}_k<F\) and \(\Psi<\tilde{\Psi}_k<G\) for $k =1,2$. Then, $\mathrm{bv}_{\tilde{\Phi}_1,\tilde{\Psi}_1}(U) = \mathrm{bv}_{\tilde{\Phi}_2,\tilde{\Psi}_2}(U)$.
    \end{enumerate}
\end{proposition}
\begin{proof}
$(i)$ For all \(\varphi \in \mathcal{U}_W(T^{F,G})\) it holds that
    \begin{align*}
        \lvert \langle \mathrm{bv}_{\tilde{\Phi},\tilde{\Psi}}(U), \varphi \rangle \rvert &= \left \lvert \int_{\Gamma^{{\tilde{\Phi}}}} U(z)\varphi(z)\mathrm{d}z + \int_{\Gamma^{{\tilde{\Psi}}}} U(z)\varphi(z)\mathrm{d}z\right\rvert \\
        &\leq \sup_{\xi\in\Gamma^{{\tilde{\Phi}}}}(e^{-w_N(\lvert\mathrm{Re}\xi\rvert)}\lvert U(\xi)\rvert) \|\varphi\|_{N+1,\tilde{\Phi},\tilde{\Psi}} \left(1+\sup_{t\in\mathbb{R}}\lvert {\tilde{\Phi}}^{\prime}(t)\rvert\right) \int_0^{\infty}e^{w_N(t)-w_{N+1}(t)}\mathrm{d}t \\
        &+ \sup_{\xi\in\Gamma^{{\tilde{\Psi}}}}(e^{-w_N(\lvert\mathrm{Re}\xi\rvert)}\lvert U(\xi)\rvert) \|\varphi\|_{N+1,\tilde{\Phi},\tilde{\Psi}}\left(1+\sup_{t\in\mathbb{R}}\lvert{\tilde{\Psi}}^{\prime}(t)\rvert\right) \int_0^{\infty}e^{w_N(t)-w_{N+1}(t)}\mathrm{d}t,
    \end{align*}
which implies the result. \\
$(ii)$ We will prove that 
$$\displaystyle\int_{\Gamma^+_{\tilde{\Phi}_1}}U(z)\varphi(z)\mathrm{d}z= \int_{\Gamma^+_{{\tilde{\Phi}_2}}} U(z)\varphi(z)\mathrm{d}z, \qquad \forall  \varphi\in\mathcal{U}_W(T^{F,G}),$$ the corresponding statement for the curves $\Gamma^-_{\tilde{\Psi}_1}$ and  $\Gamma^-_{\tilde{\Psi}_2}$  can be shown similarly. 
Set $\tilde{\Phi} = \max\{ \tilde{\Phi}_1, \tilde{\Phi}_2 \} \in \mathcal{F}(\mathbb{R})$ and note $\tilde{\Phi} < F$. By Remark \ref{stupid-remark}$(ii)$, there is $a \in (0,1)$ such that $\tilde{\Phi}(t) \leq a F(t)$ for all $t \in \mathbb{R}$. Lemma \ref{lipschitz} implies that there is  \(\Xi \in \mathcal{F}_1(\mathbb{R})\)  such that  \(\tilde{\Phi}<\Xi<F\)  and  thus \(\tilde{\Phi}_k<\Xi<F\) for $k =1,2$. Hence, we may assume without loss of generality that \(\tilde{\Phi}_1<\tilde{\Phi}_2<F\). For \(j\in\mathbb{N}\)  and $k =1,2$ we define the curves
        \begin{align*}
        \Gamma^{\tilde{\Phi}_k}_{j} = \{t+i\tilde{\Phi}_k(t)\mid -j\leq t\leq j\}, \qquad    \Gamma^{\pm}_{j}= \{\pm j +iy\mid {\tilde{\Phi}_1}(t)\leq y\leq\tilde{\Phi}_2(t)\}.
        \end{align*}
          Consider the closed curve $\Gamma_j = \Gamma^-_j\cup\Gamma^{\tilde{\Phi}_1}_j\cup\Gamma^+_j\cup\Gamma^{\tilde{\Phi}_2}_j$ and orient it counterclockwise. Let  \(\varphi\in\mathcal{U}_W(T^{F,G})\) be arbitrary. Note that there is $A > 0$ such that for all $j \in \mathbb{N}$
    \[
    \int_{\Gamma^\pm_j} \lvert U(z)\rvert \lvert \varphi(z)\rvert \lvert \mathrm{d}z\rvert \leq A e^{w_N(j)-w_{N+2}(j)}.
    \]
         Lemma \ref{N for sequences} implies that
         $$
         \lim_{j \to \infty} \int_{\Gamma^\pm_j} U(z) \varphi(z) \mathrm{d}z =0.
         $$
Hence,
        \[
        \int_{\Gamma^+_{\tilde{\Phi}_2}} U(z)\varphi(z)\mathrm{d}z - \int_{\Gamma^-_{\tilde{\Phi}_1}} U(z)\varphi(z) \mathrm{d}z = \lim_{j\to\infty} \int_{\Gamma_j}U(z)\varphi(z)\mathrm{d}z =0,
    \]
 where the last equality follows from  Cauchy's integral theorem. 
    \end{proof}
\begin{definition}\label{boundary value operator}
 Let \(\Phi,\Psi\in \mathcal{F}(\mathbb{R})\) with \(\Phi<F\) and \(\Psi<G\), and let  \(N\in\mathbb{N}\). For \(U\in \mathcal{H}_{N}(T^h\setminus \overline{T^{\Phi,\Psi}})\) we define
\[
\langle \mathrm{bv}(U), \varphi\rangle = \langle \mathrm{bv}_{\tilde{\Phi},\tilde{\Psi}}(U), \varphi\rangle =  \int_{\Gamma^+_{\tilde{\Phi}}}U(z)\varphi(z)\mathrm{d}z+\int_{\Gamma^-_{\tilde{\Psi}}}U(z)\varphi(z)\mathrm{d}z, \qquad \varphi \in \mathcal{U}_W(T^{F,G}),
\]
where \(\tilde{\Phi},\tilde{\Psi}\in\mathcal{F}_1(\mathbb{R})\) are such that \(\Phi<\tilde{\Phi}<F\) and \(\Psi<\tilde{\Psi}<G\) (such $\tilde{\Phi}$ and $\tilde{\Psi}$ exist by Lemma \ref{lipschitz}).  Proposition \ref{bvbefore} implies that $\mathrm{bv}(U) \in (\mathcal{U}_W(T^{F,G}))'$ and that the definition does not depend on the choice of  \(\tilde{\Phi}\) and \(\tilde{\Psi}\). 
\end{definition}
Finally, we explain that an element in \((\mathcal{U}_W(T^{F,G}))^{\prime}\) is the boundary value of its weighted Cauchy transform.
\begin{proposition}\label{reconstruction formula}
Let \(P\in\mathcal{U}_W(\mathbb{C})\) with \(P(0)=1\).  Then, $\mathrm{bv}(\mathrm{C}_P(f))=f$ for all  \(f\in (\mathcal{U}_W(T^{F,G}))^{\prime}\).
\end{proposition}
\begin{proof}
Let \(f\in(\mathcal{U}_W(T^{F,G}))^{\prime}\) be arbitrary. Choose \(\Phi,\Psi\in\mathcal{F}_1(\mathbb{R})\) and \(N\in \mathbb{N}\)  such that \(f \in (\mathcal{U}_W(T^{F,G}))^{\prime}_{N,\Phi,\Psi}\). By Lemma \ref{lipschitz}, there are \(\tilde{\Phi},\tilde{\Psi}\in\mathcal{F}_1(\mathbb{R})\) with \(\Phi<\tilde{\Phi}<F\) and \(\Psi<\tilde{\Psi}<G\).
Let \(\varphi\in \mathcal{U}_W(T^{F,G})\) be arbitrary. For \(\xi\in \overline{T^{\Phi,\Psi}}\) we define
\begin{equation}\label{vector-valued integrals0}
I_+(\xi) = \frac{1}{2\pi i}\int_{\Gamma^+_{\tilde{\Phi}}}\varphi(z)\frac{P(z-\xi)}{z-\xi}\mathrm{d}z, \qquad I_-(\xi)= \frac{1}{2\pi i}\int_{\Gamma^-_{\tilde{\Psi}}}\varphi(z)\frac{P(z-\xi)}{z-\xi}\mathrm{d}z.
\end{equation}
By using a similar argument as in the proof of Proposition \ref{bvbefore} but now using Cauchy's integral formula and $P(0) =1$  instead of Cauchy's integral theorem, we obtain that these integrals are convergent and that for all \(\xi\in \overline{T^{\Phi,\Psi}}\)
\begin{equation}\label{vector-valued integrals}
\varphi(\xi)=I_+(\xi)+I_-(\xi).
\end{equation}
Note that the function \(\displaystyle\Gamma^+_{j,\tilde{\Phi}}\cup\Gamma^-_{j,\tilde{\Psi}}\rightarrow  \mathcal{U}_N(\overline{T^{\Phi,\Psi}}), \, z \mapsto \varphi(z)\frac{P(z-\cdot)}{z-\cdot}\) is continuous. Moreover, as \(W\) satisfies \((\tilde{\alpha})\) and \((\tilde{N})\), it holds that
$$
\int_{\Gamma^+_{\tilde{\Phi}}}|\varphi(z)| \left \| \frac{P(z-\cdot)}{z-\cdot} \right\|_{N,\Phi,\Psi} |\mathrm{d}z| < \infty, \qquad \int_{\Gamma^-_{\tilde{\Psi}}}|\varphi(z)| \left \| \frac{P(z-\cdot)}{z-\cdot} \right\|_{N,\Phi,\Psi} |\mathrm{d}z| < \infty.
$$
Therefore, the integrals in \eqref{vector-valued integrals0} exist as $\mathcal{U}_N(\overline{T^{\Phi,\Psi}})$-valued Bochner integrals and we can view \eqref{vector-valued integrals} as an inequality in  $\mathcal{U}_N(\overline{T^{\Phi,\Psi}})$.
Since \(f\in(\mathcal{U}_W(T^{F,G}))^{\prime}_{N,\Phi,\Psi}\), the Hahn-Banach theorem implies that there is \(\tilde{f} \in  (\mathcal{U}_N(\overline{T^{\Phi,\Psi}}))'\) such that \(\tilde{f}_{\mid\mathcal{U}_W(T^{F,G})} = f\). From the proof of Lemma \ref{unique analytic extension}, it  follows that
$$
\mathrm{C}_P(f)(z) = \frac{1}{2\pi i}\left\langle \tilde{f}(\xi), \frac{P(z-\xi)}{z-\xi}\right\rangle, \qquad  z \in \mathbb{C}\setminus \overline{T^{\Phi,\Psi}}.
$$ 
We get that
\begin{align*}
\langle f,\varphi\rangle &= \langle \tilde{f},\varphi\rangle  = \langle \tilde{f}, I_+\rangle+ \langle \tilde{f},I_-\rangle\\
&= \frac{1}{2\pi i}\int_{\Gamma^+_{\tilde{\Phi}}}\varphi(z)\left\langle \tilde{f}(\xi),\frac{P(z-\xi)}{z-\xi}\right\rangle \mathrm{d}z + \frac{1}{2\pi i}\int_{\Gamma^-_{\Psi}}\varphi(z) \left\langle \tilde{f}(\xi),\frac{P(z-\xi)}{z-\xi}\right\rangle \mathrm{d}z \\
&= \int_{\Gamma^+_{\tilde{\Phi}}} \varphi(z) \mathrm{C}_P(f)(z)\mathrm{d}z + \int_{\Gamma^-_{\tilde{\Psi}}} \varphi(z) \mathrm{C}_P(f)(z)\mathrm{d}z  = \langle \mathrm{bv}(\mathrm{C}_P(f)),\varphi\rangle.
\end{align*}
\end{proof}
\subsubsection{A log-convex inequality for certain weighted supremum norms of holomorphic functions}
We start  with a generalization of the classical Hadamard three-circles theorem. 
\begin{lemma}\label{Hadamard three circle}
 Let \(\Omega_1, \Omega_2, \Omega_3\) be relatively compact open subsets of \(\mathbb{C}\) with \(\emptyset \not = \Omega_1 \subseteq \Omega_2\), \(\overline{\Omega_2} \subseteq \Omega_3\), and \(\Omega_3\) simply connected. There exists \(\theta \in (0, 1)\) such that
 \begin{equation*}
\sup_{z\in \Omega_2} \lvert f(z)\rvert \leq \left( \sup_{z\in\Omega_1} \lvert f(z)\rvert\right)^{\theta} \left(\sup_{z\in \Omega_3} \lvert f(z)\rvert\right)^{1-\theta}, \qquad  \forall f \in \mathcal{H}(\Omega_3).
 \end{equation*}
\end{lemma}
\noindent Lemma \ref{Hadamard three circle} is a special case of a result of Vogt \cite[Satz 5.1]{Hadamard}. For the reader's convenience, we include a short proof of this special case.
\begin{proof}
Fix \(x \in \Omega_1\). 
By the Riemann mapping theorem, there is a biholomorphic map \(\Phi \colon \Omega_3 \rightarrow B(0,1)\) with \(\Phi(x)=0\). 
Note that \(\Phi(\overline{\Omega_2})\subseteq B(0,1)\) is compact and that \(\Phi(\Omega_1)\) is an open neighborhood of \(0\). Hence, we can choose \(0< r < R<1\) such that \(\Phi(\overline{\Omega_2})\subseteq \overline{B}(0,R)\) and \(\overline{B}(0,r) \subseteq \Phi(\Omega_1)\). Fix \(R<s<1\) and set \(\theta = (\log s - \log R)/(\log s - \log r) \in (0,1)\). Let \(f \in \mathcal{H}(\Omega_3)\) be arbitrary and define \(f_{\Phi}= f\circ \Phi ^{-1}\in \mathcal{H}(B(0,1))\). The Hadamard three-circles theorem and the maximum principle for holomorphic functions imply that
\begin{align*}
    &\sup_{z\in \Omega_2} \lvert f(z)\rvert = \sup_{z\in \Omega_2} \lvert f_{\Phi}(\Phi(z))\rvert
    \leq \sup_{w\in \overline{B}(0,R)} \lvert f_{\Phi}(w)\rvert 
    = \sup_{w\in \partial B(0,R)} \lvert f_{\Phi}(w)\rvert \\
    &\leq \left(\sup_{w\in \partial B(0,r)} \lvert f_\Phi(w)\rvert\right)^{\theta} \left(\sup_{w\in\partial B(0,s)} \lvert f_\Phi(w)\rvert\right)^{1-\theta} \leq \left(\sup_{w\in \Phi(\Omega_1)} \lvert f_\Phi(w)\rvert\right)^{\theta} \left(\sup_{w\in B(0,1)} \lvert f_\Phi(w)\rvert\right)^{1-\theta} \\
    &= \left(\sup_{z\in\Omega_1}\lvert f(z)\rvert\right)^{\theta} \left(\sup_{z\in \Omega_3} \lvert f(z)\rvert\right)^{1-\theta}.
\end{align*}
\end{proof}
Given \(F \in \mathcal{F}(\mathbb{R})\) and \(h > \sup_{t\in \mathbb{R}}F(t)\), we define
    \[
    \Lambda^{h,\pm}_{F} =  \{z\in \mathbb{C}\mid F(\mathrm{Re}z) < \pm \mathrm{Im}z < h\},  \qquad  \Lambda^\pm_F = \{z\in \mathbb{C}\mid \pm \mathrm{Im}z > F(\mathrm{Re}z)\}. 
    \]
\begin{proposition}\label{holomorphic log-convex}
Let \(F_1,F_2,F_3 \in \mathcal{F}_1(\mathbb{R})\) with \(F_1<F_2<F_3\) and let \(W = (w_N)_{N\in\mathbb{N}}\) be a weight function system satisfying \((\tilde{\alpha})\) and \((\tilde{\omega})\). There is \(h > \sup_{t\in\mathbb{R}}F_3(t)\) such that  for all \(N \in \mathbb{N}\) and \(K \geq N+2\) there are \(\theta \in (0,1)\) and \(C > 0\) such that for all $f \in \mathcal{H}(\Lambda^{h,\pm}_{F_1})$
    \begin{equation*}
    \sup_{\xi \in \Gamma^\pm_{F_2}} e^{-w_{N+2}(\lvert \mathrm{Re}\xi\rvert)}\lvert f(\xi)\rvert \leq C \left(\sup_{\xi \in \Lambda^{h,\pm}_{F_1}} e^{-w_N(\lvert \mathrm{Re}\xi\rvert)} \lvert f(\xi)\rvert\right)^{1-\theta} \left(\sup_{\xi \in \Lambda^{h,\pm}_{F_3}} e^{-w_K(\lvert \mathrm{Re}\xi\rvert)} \lvert f(\xi)\rvert\right)^{\theta}.
    \end{equation*}
\end{proposition}
\begin{proof}
We only show the \enquote{$+$}-case, as the \enquote{$-$}-case is similar.
Fix \(0<a<\inf_{t\in \mathbb{R}}(F_2(t) - F_1(t))\) and \(b > \sup_{t\in \mathbb{R}}(F_3(t)-F_2(t))\), and set \(h = b+3+\sup_{t\in\mathbb{R}}F_2(t)\). For \(\beta \in (0,\pi/2)\) we define the sector 
    \[
    S_{\beta} = \{z \in \mathbb{C}\setminus \{0\} \mid \mathrm{Arg} (z) \in (\pi/2- \beta, \pi/2 +  \beta)\}.
    \]
    Since $F_1,F_2,F_3$ are uniformly Lipschitz continuous on \(\mathbb{R}\), there is \(\beta \in (0, \pi / 2)\) such that for $k = 1,2,3$
$$\{t+iF_k(t)\mid t\in \mathbb{R}\} + S_{\beta} \subseteq \Lambda^+_{F_k}. 
$$
    We define the triangles
    \begin{align*}
    &\Delta_1 = \{z \in S_{\beta} \mid \mathrm{Im} z < a+b+3\} - ia, \\
    &\Delta_2 = \{z \in S_{\beta} \mid \mathrm{Im}z < b + 2\}, \\
    & \Delta_3 = \{z\in S_{\beta} \mid \mathrm{Im}z < 1\} + ib.
    \end{align*}
    Note that $ \Delta_3 \subseteq \Delta_2$ and $\overline{\Delta}_2 \subseteq \Delta_1$. Lemma \ref{Hadamard three circle} implies that there is \(\theta_1 \in (0,1)\) such that
    \begin{equation}
    \label{triangles}
    \sup_{z\in \Delta_2} \lvert g(z)\rvert \leq \left(\sup_{z\in \Delta_1} \lvert g(z)\rvert\right)^{1-\theta_1} \left( \sup_{z\in\Delta_3} \lvert g(z)\rvert\right)^{\theta_1}, \qquad g \in \mathcal{H}(\Delta_1).
    \end{equation}
For \(\xi \in \Gamma_{F_2}\) and $k =1,2,3$ we define $\Delta_{\xi,k} = \Delta_{k} + \xi$. Then,  \(\Delta_{\xi,1}\subseteq \Lambda^{h,+}_{F_1}\) and \(\Delta_{\xi,3} \subseteq \Lambda^{h,+}_{F_3}\).
 Inequality  \eqref{triangles} (with  \(g(z) =f(z+\xi)\)) gives that
    \[
  |f(\xi)|\leq \left(\sup_{z\in \Delta_{\xi,1}} \lvert f(z)\rvert\right)^{1-\theta_1} \left( \sup_{z\in\Delta_{\xi,3}} \lvert f(z)\rvert\right)^{\theta_1} , \qquad  \forall f \in \mathcal{H}(\Lambda^{h,+}_{F_1}), \xi \in  \Gamma^+_{F_2}.    \]
Now let \(N\in \mathbb{N}\) and \(K > N+1\) be arbitrary. Since \(W\) satisfies \((\tilde{\omega})\), there are \(\theta_2\in(0,1)\) and \(D>0\) such that
\[
 e^{-w_{N+2}(t)} \leq D e^{-(1-\theta_2)w_{N+1}(t)}e^{-\theta_2 w_{K+1}(t)}, \qquad \forall t \geq 0.
\]
Set  \(\theta = \min(\theta_1,\theta_2)\). Let $A_L >1$, $L \in \mathbb{N}$,  be as in Lemma \ref{wsuba} 
  and set \(B = \sup_{z\in \Delta_1}\lvert \mathrm{Re}z\rvert\). Then, for all $L \in \mathbb{N}$
      \begin{align*}
 w_{L}(\lvert \mathrm{Re} z \rvert) & \leq w_{L+1}(\lvert \mathrm{Re} \xi\rvert)   + w_{L+1}(B) + \log A_L, \qquad \forall \xi  \in \Gamma^+_{F_2}, z \in \Delta_{\xi,1}.
    \end{align*}
We find that for all $f \in \mathcal{H}(\Lambda^{h,+}_{F_1})$, and $\xi \in \Gamma^+_{F_2}$
    \begin{align*}
        &e^{-w_{N+2}(\lvert \mathrm{Re}\xi\rvert)}\lvert f(\xi)\rvert \\ &\leq D e^{-(1-\theta_2)w_{N+1}(\lvert \mathrm{Re}\xi\rvert)}e^{-\theta_2 w_{K+1}(\lvert \mathrm{Re}\xi\rvert)}  \left(\sup_{z\in \Delta_{\xi,1}} \lvert f(z)\rvert\right)^{1-\theta_1} \left( \sup_{z\in\Delta_{\xi,3}} \lvert f(z)\rvert\right)^{\theta_1}  \\
   &\leq D\left(e^{-w_{N+1}(\lvert \mathrm{Re}\xi\rvert)}\sup_{z\in \Delta_{\xi,1}} \lvert f(z)\rvert\right)^{1-\theta} \left( e^{- w_{K+1}(\lvert \mathrm{Re}\xi\rvert)}  \sup_{z\in\Delta_{\xi,3}} \lvert f(z)\rvert\right)^{\theta}  \\
        &\leq C  \left(\sup_{z\in\Lambda^{h,+}_{F_1}}e^{-w_N(\lvert \mathrm{Re}z\rvert)}\lvert f(z)\rvert\right)^{1-\theta}\left(\sup_{z\in\Lambda^{h,+}_{F_3}} e^{-w_K(\lvert \mathrm{Re}z\rvert)}\lvert f(z)\rvert\right)^{\theta},
    \end{align*}
    where  \(C = DA_K^{\theta} A_N^{1-\theta}e^{\theta w_{K+1}(B) + (1-\theta) w_{N+1}(B)}\). 
    \end{proof}
\subsubsection{Proof of the result} We have all the ingredients to show the sufficiency of condition $(\omega)$ in Theorem \ref{theorem Omega}. 
\begin{proof}[Proof of Theorem \ref{theorem Omega} (sufficiency of condition \((\omega)\))]
By Lemma \ref{reduction}, we may assume without loss of generality that  \(W\) satisfies \((\tilde{\alpha})\), \((\tilde{N})\), and  \((\tilde{\omega})\). Since \(\mathcal{U}_W(T^{F,G})\not=\{0\}\), Theorem \ref{Summary surjectivity} and Remark \ref{remark non-triviality} yield that there is \(P \in \mathcal{U}_W(\mathbb{C})\) with \(P(0)=1\). Let \((a_N)_{N\in\mathbb{N}}\) be a positive increasing sequence that converges to \(1\).  By Lemma \ref{lipschitz}, there are \(F_N,G_N \in \mathcal{F}_1(\mathbb{R})\), $N \in \mathbb{N}$, such that \(a_NF<F_N<a_{N+1}F\) and \(a_NG<G_N<a_{N+1}G\) for all \(N\in\mathbb{N}\). It suffices to show that for all \(N\in\mathbb{N}\) and \(K\geq N+4\) there are \(\theta\in(0,1)\) and \(C>0\) such that 
  \[
\|f\|^*_{N+4,F_{N+4},G_{N+4}}\leq C(\|f\|^*_{N,F_N,G_N})^{1-\theta} (\|f\|^*_{K,F_K,G_K})^{\theta}, \qquad    \forall f \in (\mathcal{U}_W(T^{F,G}))^{\prime}_{N,F_N,G_N}.
  \]
Set $\Gamma_N = \Gamma^+_{F_{N}} \cup \Gamma^-_{G_{N}}$ for $N \in \mathbb{N}$.  Proposition \ref{holomorphic log-convex} implies that  for all \(N \in \mathbb{N}\) and \(K \geq N+3\) there are \(h > 0\) with \(\overline{T^{F,G}}\subseteq T^h\), \(\theta \in (0,1)\), and \(C_1 > 0\) such that for all $g \in \mathcal{H}(T^h \backslash \overline{T^{F_{N+1},G_{N+1}}})$
    \begin{equation}
    \label{almostomega}
    \sup_{\xi \in \Gamma_{N+3}} e^{-w_{N+3}(\lvert \mathrm{Re}\xi\rvert)}\lvert g(\xi)\rvert \leq C_1 \left( | g |_{N+1,F_{N+1},G_{N+1}} \right)^{1-\theta} \left(| g |_{K+1,F_{K+1},G_{K+1}} \right)^{\theta}.
\end{equation}
Let \(N\in\mathbb{N}\) and \(K\geq N+4\) be arbitrary.  Choose \(h>0\), \(\theta \in (0,1)\), and \(C_1 > 0\) such that \eqref{almostomega} holds. Propositions \ref{continuity weighted Cauchy transform}, \ref{bvbefore}$(i)$, and \ref{reconstruction formula}  imply that there are $C_2,C_3 >0$ such that for all  $f \in (\mathcal{U}_W(T^{F,G}))^{\prime}_{N,F_N,G_N}$ 
  \begin{align*}
    \|f\|^*_{N+4,F_{N+4},G_{N+4}} &= \|\mathrm{bv}(\mathrm{C}_P(f))\|^*_{N+4,F_{N+4},G_{N+4}} \\
      &\leq C_2  \sup_{\xi \in \Gamma_{N+3}}e^{-w_{N+3}(\lvert \mathrm{Re}\xi\rvert)}\lvert \mathrm{C}_P(f)(\xi)\rvert  \\
      &\leq C_1C_2 \left( | \mathrm{C}_P(f) |_{N+1,F_{N+1},G_{N+1}} \right)^{1-\theta} \left(| \mathrm{C}_P(f) |_{K+1,F_{K+1},G_{K+1}} \right)^{\theta} \\
      &\leq C_1C_2C_3(\|f\|^*_{N,F_N,G_N})^{1-\theta} (\|f\|^*_{K,F_K,G_K})^{\theta}.
  \end{align*}
  \end{proof}
\subsection{Necessity of the condition $(\omega)$} We now establish the necessity of the condition $(\omega)$ in Theorem \ref{theorem Omega}.
\begin{proof}[Proof of Theorem \ref{theorem Omega} (necessity of condition $(\omega)$)] 
By Lemma \ref{reduction}, we may assume without loss of generality that  \(W\) satisfies \((\tilde{\alpha})\) and \((\tilde{N})\). Since \(\mathcal{U}_W(T^{F,G})\not=\{0\}\), Theorem \ref{Summary surjectivity} and Remark \ref{remark non-triviality} yield that there is \(P \in \mathcal{U}_W(\mathbb{C})\) with \(P(0)=1\). 
Let $A_N >1$, $N \in \mathbb{N}$,  be as in Lemma \ref{wsuba}.
Let \((a_N)_{N\in\mathbb{N}}\) be a positive increasing sequence that converges to \(1\).
Take \(N\in\mathbb{N}\)  arbitrary. As \(\mathcal{U}_W(T^{F,G})\) satisfies \((\Omega)\), there is \(M \geq N\) such that for all \(K \geq M\), there are \(C>0\) and \(\theta\in(0,1)\) such that
\begin{equation}\label{dirac omega}
\|f \|^*_{M,a_M}\leq C(\|f \|^*_{N,a_N})^{1-\theta}(\|f \|^*_{K,a_K})^{\theta}, \qquad \forall f  \in (\mathcal{U}_W(T^{F,G}))'.
\end{equation}
Let $h > 0$ be such that $\overline{T^{F,G}} \subseteq T^h$. There is  $C_1 > 0$  such that 
$$
|P(\xi)| \leq C_1 e^{-w_{M+1}(|\mathrm{Re}\xi|)}, \qquad \forall \xi \in T^h.
$$
Let $t \in \mathbb{R}$ be arbitrary. Note that for all \(L\in\mathbb{N}\)
\[
\|\delta_t\|^* _{L,a_L} = \sup_{\varphi\in\mathcal{U}_W(T^{F,G}),\|\varphi\|_{L,a_L}\leq1}\lvert\varphi(t)\rvert\leq e^{-w_L(t)}.
\]
Hence, using \eqref{dirac omega} with $f = \delta_t$, we find that
\begin{align*}
    &e^{-w_{M+1}(t)}=e^{-w_{M+1}(t)}\lvert \langle \delta_t,P(\cdot-t)\rangle\rvert 
    \leq e^{-w_{M+1}(t)}\|\delta_t\|^*_{M,a_M} \|P(\cdot-t)\|_{M,a_M} \\
    &\leq A_MC_1 \|\delta_t\|^*_{M,a_M}\leq A_MCC_1\left(\|\delta_t\|^*_{N,a_N}\right)^{1-\theta}\left(\|\delta_t\|^*_{K,a_K}\right)^{\theta} \\
    &\leq A_MCC_1e^{-(1-\theta)w_N(t)}e^{-\theta w_K(t)},
\end{align*}
which implies that \(W\) satisfies \((\omega)\).
\end{proof}

\section{Proof of Theorem \ref{Converse}}\label{Proof converse}

In this section, we show Theorem \ref{Converse}.  Our proof of is based on the following result about interpolation of vector-valued holomorphic functions from  \cite{D-N}.
Given $X \subseteq \mathbb{C}$ open and a lcHs $E$, we denote by $\mathcal{H}(X;E)$ the space of $E$-valued holomorphic functions on $X$.

\begin{proposition} \cite[Example 7.8]{D-N}\label{Concrete interpolation}
    Let \(X\subseteq \mathbb{C}\) be  open and connected, and let \(E\) be a locally complete lcHs with a fundamental sequence of bounded sets. Suppose that there is a sequence \((z_n)_{n\in\mathbb{N}}\) in \(X\) without accumulation points in \(X\) such that for each sequence \((e_n)_{n\in\mathbb{N}}\) in \(E\) there exists  \(\mathbf{f} \in \mathcal{H}(X;E)\) with \(\mathbf{f}(z_n)=e_n\) for all \(n\in\mathbb{N}\). Then, \(E\) satisfies \((A)\).
\end{proposition}

Hence, to prove Theorem \ref{Converse}, it suffices to show that the surjectivity of the $E$-valued Cauchy--Riemann operator \(\overline{\partial}\colon \mathcal{K}_W(T^{F,G};E)\rightarrow\mathcal{K}_W(T^{F,G};E)\) implies that there is a sequence \((z_n)_{n\in\mathbb{N}}\) in \(T^{F,G}\), without accumulation points in \(T^{F,G}\), such that for each sequence \((e_n)_{n\in\mathbb{N}}\) in \(E\) there exists  \(\mathbf{f} \in \mathcal{H}(T^{F,G};E)\) with \(\mathbf{f}(z_n)=e_n\) for all \(n\in\mathbb{N}\). This is established in the following result, which is inspired by \cite[Theorem 3]{Interpolation theory}.
\begin{lemma}\label{Interpolation theorem}
    Let \(F,G\in\mathcal{F}(\mathbb{R})\), let \(W\) be a weight function system, and let \(E\) lcHs. Let \((t_n)_{n\in\mathbb{N}}\) be an increasing sequence in \((0,F(0))\) such that \(t_n \to F(0)\) as $n \to \infty$. Suppose  that the map \(\overline{\partial}\colon \mathcal{K}_W(T^{F,G};E)\rightarrow\mathcal{K}_W(T^{F,G};E)\) is surjective. Then, for each sequence \((e_n)_{n\in\mathbb{N}}\) in \(E\) there exists \(\mathbf{f} \in \mathcal{H}(T^{F,G};E)\) such that \(\mathbf{f} (it_n)=e_n\) for all \(n\in\mathbb{N}\).
    \end{lemma}
    \begin{proof}
    Note that \((it_n)_{n\in\mathbb{N}}\) is a sequence in \(T^{F,G}\) without accumulation points in \(T^{F,G}\). Choose  a sequence \((\varepsilon_n)_{n\in\mathbb{N}}\) in \((0,1)\) such that \(\{it_k \mid k \in \mathbb{N}\}\cap B(it_n,\varepsilon_n) = \{it_n\}\) for all \(n\in\mathbb{N}\). Take \(\psi_n \in C^{\infty}(\mathbb{R}^2)\) such that \(\psi_n \equiv 1\) on \(\overline{B}(it_n,\varepsilon_n/2)\) and  \(\mathrm{supp}\psi_n\subseteq B(it_n,\varepsilon_n)\). 
    Let \((e_n)_{n\in\mathbb{N}}\) be an arbitrary sequence in \(E\). Define \(\mathbf{g}(z) = \sum_{n=1}^{\infty} e_n \psi_n(z)\) for \(z\in\mathbb{C}\). Then,  \(\mathbf{g} \in C^{\infty}(T^{F,G};E)\), as the sum defining \(\mathbf{g}\)  is locally finite, and 
    \begin{equation}
    \label{compact}
    \lvert \mathrm{Re} z\rvert > 1 \implies \mathbf{g}(z)=0, \qquad \forall z \in \mathbb{C}.
    \end{equation}
Moreover,  \(\mathbf{g}(it_n)=e_n\)  for all \(n \in \mathbb{N}\).
          Choose \(\varphi\in\mathcal{H}(T^{F,G})\) such that $\varphi$ has a zero at each point $it_n$, $n \in \mathbb{N}$, and such that \(\varphi\) has no other zeros in \(T^{F,G}\); see e.g.\ \cite[Theorem 2]{Interpolation theory}.
    The idea is now to define \(\mathbf{f}=\mathbf{g}+\varphi \mathbf{h}\) for a suitably chosen \(\mathbf{h} \in C^{\infty}(T^{F,G};E)\) such that \(\mathbf{f}\in\mathcal{H}(T^{F,G},E)\). Then, \(\mathbf{f}(it_n)= \mathbf{g}(it_n)=e_n\) for all \(n\in\mathbb{N}\) and the proof would be complete. We now show that such an \(\mathbf{h}\) exists. It suffices to show that there is  \(\mathbf{h}\in C^\infty(T^{F,G};E)\) such that \(\overline{\partial}(\mathbf{g}+\varphi \mathbf{h}) = 0\) on \(T^{F,G}\), or equivalently, as \(\varphi\in\mathcal{H}(T^{F,G})\),  \(\varphi  \overline{\partial} \mathbf{h} =- \overline{\partial} \mathbf{g} \) on \(T^{F,G}\).
    Since \(\varphi\) has the points \(it_n\), $n\in\mathbb{N}$, as its only zeros in \(T^{F,G}\), and  \(\overline{\partial}\mathbf{g} = 0\) in a neighborhood of \(it_n\) for all \(n\in\mathbb{N}\),  it follows that \(\overline{\partial}\mathbf{g}/\varphi\in C^{\infty}(T^{F,G};E)\). In fact, by \eqref{compact},  
      it holds that \(\overline{\partial}\mathbf{g}/\varphi \in \mathcal{K}_W(T^{F,G};E)\). As \(\overline{\partial}\) is surjective on \(\mathcal{K}_W(T^{F,G};E)\), there is  \(\mathbf{h} \in  \mathcal{K}_W(T^{F,G};E) \subseteq C^{\infty}(T^{F,G};E)\) such that \(\overline{\partial} \mathbf{h} = -\overline{\partial}\mathbf{g}/\varphi\)  and thus \(\varphi\overline{\partial} \mathbf{h} =- \overline{\partial} \mathbf{g} \) on \(T^{F,G}\). 
    \end{proof}
    
\begin{proof}[Proof of Theorem \ref{Converse}]
The result follows from  Proposition \ref{Concrete interpolation} and Lemma \ref{Interpolation theorem}.
\end{proof}

\end{document}